\documentclass[11pt,english,a4paper]{smfart}
\usepackage{vaucanson-g}
\usepackage{graphicx}
\usepackage{ae,amsfonts,euscript,enumerate}

\renewcommand{\theequation}{\thesection.\arabic{equation}}

\newtheorem{thm}{Theorem}[section]
\newtheorem{exam}{Example}[section]
\newtheorem{defn}{Definition}[section]

\newtheorem{prop}{Proposition}[section]
\newtheorem{lem}{Lemma}[section]
\newtheorem{cor}{Corollary}[section]

\newtheorem{rem}{Remark}[section]

\begin{document}
\bibliographystyle{plain}
\title[]{On vanishing coefficients of algebraic power series over fields of positive characteristic}
\author{Boris Adamczewski}
\address{
CNRS, Universit\'e de Lyon, Universit\'e Lyon 1\\
Institut Camille Jordan  \\
43 boulevard du 11 novembre 1918 \\
69622 Villeurbanne Cedex, France}
\email{Boris.Adamczewski@math.univ-lyon1.fr}

\author{Jason P.~Bell}
\thanks{The First author was supported by ANR grants Hamot and SubTile. The second author was supported by NSERC grant 31-611456.}

\address{
Department of Mathematics\\
Simon Fraser University\\
Burnaby, BC, Canada\\
 V5A 1S6}

\email{jpb@math.sfu.ca}




\bibliographystyle{plain}

\begin{abstract}  Let $K$ be a field of characteristic $p>0$ and let $f(t_1,\ldots ,t_d)$ be a power series in $d$ variables with coefficients in $K$  that is algebraic over the field of multivariate rational functions $K(t_1,\ldots ,t_d)$.  We prove a generalization of both Derksen's recent analogue of the Skolem--Mahler--Lech theorem in positive characteristic and a classical theorem of Christol, by showing that the set of indices $(n_1,\ldots ,n_d)\in \mathbb{N}^d$ for which the coefficient of $t_1^{n_1}\cdots t_d^{n_d}$ in $f(t_1,\ldots ,t_d)$ is zero is a $p$-automatic set.  Applying this result to multivariate rational functions  
leads to interesting effective results concerning some Diophantine equations related to 
$S$-unit equations and more generally to the Mordell--Lang Theorem over fields of  positive characteristic. 
\end{abstract}
\maketitle

\tableofcontents
\section{Introduction}\label{introduction}
The Skolem--Mahler--Lech theorem is a celebrated result which describes the set of solutions in 
$n$ to the equation $a(n)=0$, where $a(n)$ is a sequence satisfying a linear recurrence over 
a field of characteristic $0$.  
We recall that if $K$ is a field and $a$ is a $K$-valued sequence, 
then $a$ satisfies a linear recurrence over $K$ if there exists a natural number 
$m$ and values $c_1,\ldots ,c_m\in K$ such that 
\begin{displaymath}
a(n)=\sum_{i=1}^m c_i a(n-i)
\end{displaymath}
 for all sufficiently large values of $n$. The zero set of the linear recurrence $a$ is defined 
 by 
 \begin{displaymath}
\mathcal Z( a) := \left\{ n \in \mathbb N \mid f(n)=0 \right\} \, .
\end{displaymath}
The Skolem--Mahler--Lech theorem can then be stated as follows.

\begin{thm}[Skolem--Mahler--Lech]
Let $a$ be a linear recurrence over a field of characteristic $0$.  
Then the set $\mathcal Z(a)$ is a union of a finite set and a finite number
of infinite arithmetic \label{thm: SMLlr} progressions. 
\end{thm}

This result was first proved for linear recurrences over the rational
numbers by Skolem \cite{Sk}.  It was next extended to linear recurrences 
over the algebraic numbers by Mahler \cite{Mah}.
The version above was proven first by Lech \cite{Lech} and later 
by Mahler \cite{Mah1, Mah2}.
More details about the history of this theorem can be found in the book by 
Everest {\it et al.} \cite{EPSW}.   

Though the conclusion of the Skolem--Mahler--Lech theorem obviously holds for linear recurrences 
defined over finite fields, this is not the case for infinite fields $K$ of positive characteristic. 
The simplest counter-example was given by Lech \cite{Lech}.  
Throughout this paper, $p$ will denote a prime number. 
Let $K=\mathbb{F}_p(t)$ be the field of rational 
functions in one variable over $\mathbb{F}_p$.  Let 
\begin{displaymath}
a(n) :=(1+t)^n-t^n-1 \, . 
\end{displaymath}
We can observe that the sequence 
$a$ satisfies the recurrence 
\begin{displaymath} 
a(n)\ = \ (2+2t)a(n-1)-(1+3t+t^2)a(n-2)+(t+t^2)a(n-3) 
\end{displaymath}
 for $n>3$, while  
\begin{displaymath} 
\mathcal Z(a) = \{1,p,p^2,p^3,\ldots\} \, .
\end{displaymath} 

\medskip

More recently, Derksen \cite{Der} gave more pathological examples,  
which show that the correct analogue of the Skolem--Mahler--Lech theorem in positive 
characteristic is much more subtle.  For example, one has 
 $$
 \mathcal Z(a) = \{p^n \mid n \in \mathbb N\} \cup  \{p^n + p^m \mid n, m \in \mathbb N\} \, ,
 $$
for the linear recurrence  
$a$ defined over the field $\mathbb{F}_p(x,y,z)$  by
 \begin{displaymath}
 a(n) := (x + y + z)^n - (x + y)^n - (x + z)^n - (y + z)^n + x^n + y^n + z^n \, .
 \end{displaymath} 
Derksen noted that while pathological examples of zero sets of  linear recurrences do exist in 
characteristic $p$, the base-$p$ expansions of the natural numbers in the zero set are still well behaved.  
In fact, he proved the remarkable result that the zero set of a linear recurrence 
can always be described in terms of finite automata \cite{Der}.
 
\begin{thm}[Derksen]\label{thm:derksen}
Let $a$ be a linear recurrence  over a field $K$ of characteristic $p$.  
Then the set $\mathcal{Z}(a)$ is $p$-automatic.
\end{thm}

\medskip

We recall that an infinite sequence $a$ with values in a finite set 
is said to be $p$-automatic if $a(n)$ is a finite-state function of the base-$p$ 
representation of $n$. Roughly, this means that there exists a finite automaton taking the
 base-$p$ expansion of $n$ as input and producing the term $a(n)$ as 
output. A set ${\mathcal E}\subset \mathbb N$ is said to be  
$p$-automatic if there exists a finite automaton that reads as input the base-$p$ expansion of 
$n$ and accepts this integer (producing as output the symbol $1$) if  $n$ belongs to 
${\mathcal E}$, otherwise this automaton rejects the integer 
$n$, producing as output the symbol $0$. 

\medskip

Let us give a formal definition of both notions.
Let $k\ge 2$ be a natural number.  We let $\Sigma_k$ denote the alphabet  
$\left\{0,1,\ldots,k-1\right\}$. A $k$-automaton\index{$k$-automaton} is a $6$-tuple 
\begin{displaymath}
{\mathcal A} = \left(Q,\Sigma_k,\delta,q_0,\Delta,\tau\right) ,
\end{displaymath}
where $Q$ is a finite set of states, 
$\delta:Q\times \Sigma_k\rightarrow Q$ is the transition function, $q_0$ is 
the initial state, $\Delta$ is the output alphabet and $\tau : Q\rightarrow 
\Delta$ is the output function. For a state $q$ in $Q$ and for a finite 
word $w=w_1 w_2 \cdots w_n$ on the alphabet $\Sigma_k$, 
we define $\delta(q,w)$ recursively by 
$\delta(q,w)=\delta(\delta(q,w_1w_2\cdots w_{n-1}),w_n)$. 
Let $n\geq 0$ be an integer and let 
$w_r w_{r-1}\cdots w_1 w_0$ in $\left(\Sigma_k\right)^{r+1}$ 
be the base-$k$ expansion 
of $n$. Thus $n=\sum_{i=0}^r w_i k^{i} :=[w_rw_{r-1}\cdots w_0]_k$. We denote by $w(n)$ the 
word $w_0 w_1 \cdots w_r$.

\begin{defn}{\em A sequence $(a_n)_{n\geq 0}$ is 
said to be $k$-automatic if there exists a $k$-automaton ${\mathcal A}$ such that 
$a_n=\tau(\delta(q_0,w(n)))$ for all $n\geq 0$.}
\end{defn}

\begin{defn}{\em A set ${\mathcal E}\subset \mathbb N$\index{automatic set} is said to be recognizable 
by a finite $k$-automaton,  
or for short $k$-automatic, if the characteristic sequence of ${\mathcal E}$, defined by 
$a_n=1$ if $n\in {\mathcal E}$ and $a_n=0$ otherwise, is a $k$-automatic sequence. }
\end{defn}

More generally, feeding a finite automaton with $d$-tuples of nonnegative integers leads to the notion of $p$-automatic subsets of $\mathbb N^d$.  Some background on automata theory, including examples, formal definitions of multidimensional automatic sequences and sets, and their extension to arbitrary finitely generated abelian groups, 
are given in Section \ref{Salon}.

\begin{rem}{\em Let us make few important remarks. 

\medskip

\begin{itemize}

\item[$\bullet$] In the previous definitions, we chose the convention that the base-$k$ expansion of $n$ is scanned from 
left to right. Our automata thus read the input starting with the most significant digit. 
We recall that it is well-known that 
the class of $k$-automatic sets or sequences remains unchanged when choosing 
to read the input starting from the least significant digit (see for instance Chapter V of \cite{Eilenberg} or 
Chapter 5 of \cite{AS}). 

\medskip

\item[$\bullet$]  One could also ask whether the base $k$ plays an important role here. 
As proved in a fundamental paper of  Cobham \cite{Cob69}, this is actually the case.  
Periodic sets, that are sets obtained as a union of a finite set and a finite number
of infinite arithmetic progressions, are exactly those that are $k$-automatic for every integer $k\geq 2$. 
In addition, an infinite aperiodic $k$-automatic set is also $k^n$-automatic for 
every positive integer $n$, while it cannot be $\ell$-automatic if $k$ and $\ell$ are two multiplicatively 
independent integers.

\medskip

\item[$\bullet$]  The class of $k$-automatic sets is closed under various natural operations such 
as intersection, union and complement (see for instance Chapter V of \cite{Eilenberg} or Chapter 5 of \cite{AS}). 
\end{itemize}}
\end{rem}

\medskip
 
On the other hand, it is well known that if $K$ is a field and $a$ is a $K$-valued sequence, then $a$ satisfies 
 a linear recurrence over $K$ if and only if the power series
 $$f(t)=\sum_{n=0}^{\infty} a(n)t^n$$ is the power series expansion of a rational function.  For instance, 
 Mahler \cite{Mah, Mah1, Mah2} worked with rational power series rather than linear recurrences 
 when proving what we now call the Skolem--Mahler--Lech theorem.  Let  
 $$
 \mathcal{Z}(f) := \{n \mid a(n) = 0 \} \, .
 $$Then Derksen's theorem 
 can be restated as follows: 
let $K$ be a field of characteristic $p$ and let $f(t)\in K[[t]]$ be a rational function,   
then the set $\mathcal{Z}(f)$ is $p$-automatic.

\medskip

This formulation of Derksen's theorem is in the same spirit as another famous result involving automata theory and known 
as Christol's theorem \cite{Christol}.   

\begin{thm}[Christol]\label{thm:christol}
Let $q$ be a positive integer power of $p$.  
Then $f(t)=\sum_{n= 0}^{\infty} a(n)t^n\in \mathbb F_q[[t]]$ 
is algebraic over $\mathbb F_q(t)$ 
if and only if the sequence $a$ is $p$-automatic.
\end{thm}

The main aim of this paper is to produce a simultaneous multivariate generalization of both 
the theorem of Derksen and the theorem of Christol.

Given a multivariate power series 
$$f(t_1,\ldots,t_d)= \sum_{(n_1,\ldots,n_d)\in\mathbb N^d}a(n_1,\ldots,n_d)
t_1^{n_1}\cdots t_d^{n_d}\in K[[t_1,\ldots,t_d]]\, ,$$ 
we define  the set of vanishing coefficients of $f$  by  
\begin{displaymath}
\mathcal Z(f) = \{ (n_1,\ldots,n_d)\in \mathbb N^d \mid a(n_1,\ldots,n_d)=0\} \,.
 \end{displaymath}

Our main result reads as follows.

\begin{thm} Let $K$ be a field of characteristic $p$ and let 
$f(t_1,\ldots ,t_d) \in K[[t_1,\ldots ,t_d]]$ be a power series that is algebraic over 
the field of multivariate rational functions $K(t_1,\ldots ,t_d)$.  Then the set 
${\mathcal Z}(f)$ is $p$-automatic.
\label{thm: main}
\end{thm}

Let us make few comments on this result.

\medskip

\begin{itemize}
\item[$\bullet$]
In the case that $d=1$ and $f(t)$ is chosen to be the power series expansion of a 
rational function in Theorem \ref{thm: main}, we immediately obtain Derksen's theorem (Theorem \ref{thm:derksen}).  We do not obtain his finer characterization, but, as explained in  
Section \ref{conclude}, it is not possible to obtain a significantly improved characterization 
of zero sets even for multivariate rational power series.  

\medskip

\item[$\bullet$] In the case that $d=1$ and $K$ is chosen to be a finite field in Theorem \ref{thm: main}, 
 we cover the more difficult direction of Christol's theorem (Theorem \ref{thm:christol}). Indeed,  if
$f(t)=\sum_{n=0}^{\infty} a(n) t^n \in K[[t]]$ is an algebraic power series, 
then for each $x\in K$ the function $f(t)-x/(1-t)$ is algebraic. Theorem \ref{thm: main} 
thus implies that the set $\{n\in \mathbb{N} \mid a(n)=x\}$ is $p$-automatic for all $x\in K$.  
This immediately implies that the sequence $a$ is $p$-automatic. 

\medskip

\item[$\bullet$] Theorem \ref{thm: main} can actually take a stronger form. 
Let ${\mathcal E}\subset \mathbb N^d$. The following 
conditions are equivalent.
\begin{itemize}
\item[\textup{(i)}]  The set ${\mathcal E}$ is $p$-automatic. 
\item[\textup{(ii)}] ${\mathcal E}={\mathcal Z(f)}$ for some algebraic power series with coefficients over 
a field of characteristic $p$. 
\end{itemize} 
Indeed, it is known \cite{Salon1987} 
that given a $p$-automatic set ${\mathcal E}\subset \mathbb N^d$, the formal power series 
$$f(t_1,\ldots,t_d)=\sum_{(n_1,\ldots,n_d)\in \mathcal E} t_1^{n_1}\cdots t_d^{n_d}$$ is algebraic over 
$\mathbb F_p(t_1,\ldots,t_d)$.  
From the latter property and Theorem \ref{thm: main}, we also deduce the following result. 
Let $K$ be a field of characteristic $p$ and let 
$$f(t_1,\ldots,t_d)= \sum_{(n_1,\ldots,n_d)\in\mathbb N^d}a(n_1,\ldots,n_d)
t_1^{n_1}\cdots t_d^{n_d}\in K[[t_1,\ldots,t_d]]\, $$ be a power series that is algebraic over 
the field of multivariate rational functions $K(t_1,\ldots ,t_d)$. For $x\in K$, let 
$$
 a^{-1}(x) := \left\{ (n_1,\ldots,n_d)\in \mathbb N^d \mid a(n_1,\ldots,n_d)=x\right\} \, .
$$
Then for every $x\in K$ the formal 
power series 
$$
f_x(t_1,\ldots,t_d):=\sum_{(n_1,\ldots,n_d)\in a^{-1}(x)} t_1^{n_1}\cdots t_d^{n_d} 
$$
is also algebraic. In the particular case where $K$ is a finite field, this result was first proved by Furstenberg  \cite{Fur} (see also the more recent result of Kedlaya \cite{Ked} for a generalization to Hahn's power series with coefficients in a finite field).  
 
\medskip

\item[$\bullet$]
No such multivariate generalization of the Skolem--Mahler--Lech theorem exists in characteristic $0$.  For example, if one takes the rational bivariate power series
$$f(x,y)=\sum_{n,m} (n^3 - 2^m) x^n y^m\in \mathbb{Q}[[x,y]]\, ,$$ then 
${\mathcal Z}(f)=\{(n,m) \mid  m\equiv 0~(\bmod~3), n=2^{m/3}\}$.  This shows that 
there is no natural way to express the set 
of vanishing coefficients of $f$ in terms of more general arithmetic progressions or in terms of automatic sets.  
In fact, finding zero sets of coefficients of multivariate rational power series with integer coefficients is often equivalent to very difficult classes 
of Diophantine problems which cannot be solved at this moment, such as for instance finding an effective procedure to solve all $S$-unit equations (see Section \ref{sunit} for more details).  In Section \ref{decidability}, we also give a Diophantine problem related to linear recurrences which is conjectured in \cite{CMP} 
to be undecidable and, as shown in the proof of Theorem \ref{thm:cmp}, which is equivalent to describe the zero sets of coefficients of a class of simple mutivatiate rational power series with integer coefficients.

\end{itemize}

\medskip

Our proof of Theorem \ref{thm: main} involves using methods of Derksen as well as more advanced 
techniques from automata theory reminiscent of works of Christol \cite{Christol}, Denef and Lipshitz \cite{DL}, Harase \cite{Har88}, 
Shariff and Woodcock \cite{SW} among  others.  
We first consider the action of a certain infinite semigroup on the ring of power series over a field of characteristic $p$.  We use the fact that algebraic power series  have a finite orbit under the action of this semigroup to 
apply Derksen's ``Frobenius splitting'' technique which allows us to show that the set of vanishing coefficients 
is necessarily $p$-automatic. An especially important aspect of the proof of Theorem \ref{thm:derksen} 
is that each step can be made effective. We prove that this is also the case with Theorem \ref{thm: main}.  

\begin{thm}\label{thm: eff} Let $K$ be a field of positive characteristic and let 
$f(t_1,\ldots ,t_d) \in K[[t_1,\ldots ,t_d]]$ be a power series that is algebraic over 
the field of multivariate rational functions $K(t_1,\ldots ,t_d)$.  Then the set 
${\mathcal Z}(f)$ can be effectively determined. Furthermore, the following 
properties are decidable. 

\medskip

\begin{itemize}
\item[{\rm (i)}] the set ${\mathcal Z}(f)$ is empty.

\medskip

\item[{\rm (ii)}] the set ${\mathcal Z}(f)$ is finite.

\medskip

\item[{\rm (iii)}] the set ${\mathcal Z}(f)$ is periodic, that is, formed by the union of a finite set and 
of a finite number of ($d$-dimensional) arithmetic progressions.

\medskip

\end{itemize}
\noindent 
In particular, when ${\mathcal Z}(f)$ is finite, one can determine (in a finite amount of time) 
all its elements.
\end{thm}

\begin{rem}{\em
When we say that the set ${\mathcal Z}(f)$ can be effectively determined, this means 
that there is an algorithm that produces a $p$-automaton 
that generates ${\mathcal Z}(f)$ in a finite amount of time.  Furthermore, there exists an algorithm that allows 
one to determine in a finite amount of time whether or not ${\mathcal Z}(f)$ is empty, finite, or periodic.} 
\end{rem}

As we will illustrate in Sections \ref{decidability}, \ref{sunit} and \ref{MLT}, applying Theorem \ref{thm: eff} 
to multivariate rational functions actually leads to interesting 
effective results concerning some Diophantine equations related to $S$-unit equations 
and more generally to the Mordell--Lang Theorem over fields of positive characteristic.

\medskip

The outline of this paper is as follows.  Our Diophantine applications are discussed in 
Sections \ref{decidability}, \ref{sunit} and \ref{MLT}. In Section \ref{Salon}, 
we recall some basic background on automata theory.  We define in particular 
the notion of automatic sets of $\mathbb N^d$ and more generally of automatic subsets 
of finitely generated abelian groups. The latter notion does not appear to have been introduced earlier and 
may be of independent interest. In Section \ref{proof}, we prove Theorem \ref{thm: main}.  
In Sections \ref{section: ker} and \ref{section: eff}, we make the proof of Theorem \ref{thm: main} effective,  
proving Theorem 
\ref{thm: eff}. Finally, we conclude our paper with some comments in Section \ref{conclude}.

\section{Linear recurrences and decidability}\label{decidability}

There are many different proofs and extensions of the Skolem--Mahler--Lech 
theorem in the literature (see for instance \cite{Bez,Han86,vdP,EPSW}). 
These proofs all use $p$-adic methods in some way, although the result is valid in
any field of characteristic $0$.  This seems to be responsible for a well-known deficiency of the 
Skolem--Mahler--Lech theorem: all known proofs are ineffective.  
This means that we do not know any algorithm that allows us to  
 determine the set ${\mathcal Z}(a)$ for a given linear recurrence $a(n)$ 
defined over a field of characteristic $0$. 
We refer the reader to \cite{EPSW} and to the recent discussion in \cite{Ta08} for more details.  
It is actually still unknown 
whether the fact that ${\mathcal Z}(a)$ is empty or not is decidable. 
In fact, it seems unclear that one should even expect it to be decidable. In this direction, let us recall 
the following conjecture from \cite{CMP}. 
Given linear recurrences $a_1(n),\ldots,a_d(n)$ over a field $K$, 
we let 
$$
{\mathcal Z}(a_1,\ldots,a_d) := 
\left\{(n_1,\ldots,n_d)\in \mathbb N^d \mid a_1(n_1)+\cdots + a_d(n_d)=0 \right\} \, .
$$
It was conjectured in \cite{CMP} that, if $K=\mathbb Q$,  
the property 
$$
{\mathcal Z}(a_1,\ldots,a_d) \not = \emptyset
$$
is undecidable for every positive integer $d$ large enough.  

\medskip

As mentioned in the introduction, the situation is drastically different for fields of 
positive characteristic. Indeed, Derksen \cite{Der} proved  
that each step of the proof of Theorem \ref{thm:derksen} can be made effective. 
In particular, there exists an algorithm that allows one to decide whether the set 
${\mathcal Z}(a)$ is empty or not in a finite amount of time. 
We give below a generalization of Derksen's theorem 
to an arbitrary number of linear recurrences. It well illustrates the relevance of 
Theorem \ref{thm: eff}. 

\begin{thm}\label{thm:cmp}
Let $K$ be a field of characteristic $p$, $d$ a positive integer, and let 
$a_1(n),\ldots,a_d(n)$ be linear recurrences over 
$K$. Then ${\mathcal Z}(a_1,\ldots,a_d)$ is a $p$-automatic set that can be effectively determined. 
In particular, the property $$
{\mathcal Z}(a_1,\ldots,a_d) \not = \emptyset
$$ 
is decidable.
\end{thm}

Note that, in addition, we can decide whether such a set ${\mathcal Z}(a_1,\ldots,a_d)$ is finite or 
periodic.  

\begin{proof}
In view of Theorem \ref{thm: eff}, it suffices to prove that there exists an explicit 
multivariate rational function $f(t_1,\ldots,t_d) \in K(t_1,\ldots,t_d)$ such that 
${\mathcal Z}(f)={\mathcal Z}(a_1,\ldots,a_d)$. 

Let $i\in\{1,\ldots,d\}$. Since $a_i$ is a linear recurrence over $K$, we have that 
$f_i(t):=\sum_{n\geq 0} a_i(n)t^n$ is a rational function. Thus, 
$$
f(t_1,\ldots,t_d) := \sum_{i=1}^d \left( f_i(t_i)\cdot \prod_{j\not=i} \frac{1}{1-t_j}\right)
$$
is a multivariate rational function in $K(t_1,\ldots,t_d)$. Furthermore, this definition 
implies that 
$$f(t_1,\ldots,t_d) = \sum_{(n_1,\ldots,n_d)\in \mathbb N^d} (a_1(n_1)+\cdots +a_d(n_d))
t_1^{n_1}\cdots t_d^{n_d} \, .$$
We thus deduce that ${\mathcal Z}(f)={\mathcal Z}(a_1,\ldots,a_d)$. This ends the proof. 
\end{proof}
\section{Linear equations over multiplicative groups}\label{sunit}

In this section, we discuss some Diophantine equations that generalize the famous 
$S$-unit equations (see for instance the survey \cite{EGST}). More precisely, given 
a field $K$ and a finitely generated subgroup $\Gamma$ of $K^*$, we consider 
linear equations of the form 
\begin{equation}\label{eq: sunit}
c_1X_1+\cdots +c_d X_d =1 \, ,
\end{equation} 
where $c_1,\ldots,c_d$ belong to $K$ and where we look for solutions in 
$\Gamma^d$. 

These equations have a long history. 
Let $S$ be a finite number of prime numbers and $\Gamma\subseteq \mathbb Q^*$ 
the multiplicative group generated by the elements of $S$. 
In 1933, Mahler \cite{Mah33} proved that for all nonzero 
rational numbers $a$ and $b$ the equation  
\begin{equation}\label{eq: 2v}
aX+bY=1
\end{equation}
has only a finite number of solutions in $\Gamma^2$.  Lang \cite{Lan60} later generalized this result 
by proving that for all $a$ and $b$ belonging to $\mathbb C^*$ and all subgroups of finite 
rank $\Gamma$ of $\mathbb C^*$, Equation (\ref{eq: 2v}) has only a finite number of solutions in 
$\Gamma^2$. Furthermore, in the case where $\Gamma$ is a subgroup of $\mathbb Q^*$, there 
exists an effective method based on the theory of linear forms of logarithms to determine all solutions of 
Equation (\ref{eq: 2v}). 

When the number of variables $d$ is larger than $2$, one can no longer 
expect that Equation (\ref{eq: sunit}) 
necessarily has only a finite number of solutions. However, the subspace theorem can be used 
to prove that such an equation has only a finite number of nondegenerate solutions; that is, 
solutions with the property that no proper subsum vanishes \cite{Ev84,PS91}. 
Furthermore, it is possible to use some quantitative version 
of the subspace theorem to bound the number of nondegenerate solutions.  
In this direction, the following general and very strong result was obtained by 
Evertse, Schlickewei and W.\. M. Schmidt \cite{ESS02}:  
given $K$ a field of characteristic $0$ and $\Gamma$ a multiplicative subgroup of 
rank $r$ of $K^*$, Equation (\ref{eq: sunit}) has at most 
$\exp ((6d)^{3d}(r+1))$
nondegenerate solutions. 
However, all general known results concerning more than two variables are ineffective. 


\medskip

The situation in characteristic $p$ is similar to the one encountered 
with the Skolem--Mahler--Lech theorem. The Frobenius endomorphism 
may be responsible for the existence of ``pathological solutions". 
Indeed, it is easy to check that, for every positive integer $q$ that is 
a power of $p$, the pair $(t^q,(1-t)^q)$ is a solution of the equation 
$$
X+Y=1
$$ in 
$\Gamma^2$, where $\Gamma$ is the multiplicative subgroup of $\mathbb F_p(t)^*$ 
generated by $t$ and $1-t$. 
In fact, if we take $K=\mathbb F_p(t)$ and $\Gamma=\langle t,(1-t)\rangle$, 
we can find more sophisticated examples. As observed in \cite{Masser}, 
the equation 
$$
X+Y-Z=1
$$
has for every pair of positive integer $(n,m)$ the nondegenerated solution 
$$
X=t^{(p^n-1)p^m},\;\; Y=(1-t)^{p^{n+m}},\;\; Z=t^{(p^n-1)p^m}(1-t)^{p^m}\, .
$$
Thus, there is no hope to obtain in this framework results similar to those mentioned 
previously. Concerning Equation (\ref{eq: 2v}), Voloch \cite{Vo98} gave interesting results. 
He obtained, in particular, conditions that ensure the finiteness of the number of solutions 
(with explicit bounds for the number of solutions). Masser \cite{Masser} obtained a result 
concerning the structure of the solutions of the general Equation (\ref{eq: sunit}). His aim was 
actually to prove a conjecture of K. Schmidt concerning mixing properties of algebraic 
$\mathbb Z^d$-actions (see \cite{Masser,Sch01,SW93} for more details on this problem). 

\medskip

As a consequence of Theorem \ref{thm: eff}, 
we are able to give a satisfactory effective solution to the general equation (\ref{eq: sunit}) 
over fields of positive characteristic, proving that the set of solutions is $p$-automatic in a natural sense. 
We note that the notion of an automatic subset of a 
finitely generated abelian group is given in Section \ref{Salon} (see Definition \ref{def:autofgg} and 
Proposition \ref{prop:fggcarac}).

\begin{thm}\label{thm: sunit}
Let $K$ be a field of characteristic $p$, let $c_1,\ldots,c_d\in K^*$, and let $\Gamma$ 
be a finitely generated multiplicative subgroup $K^*$. Then the set of solutions in 
$\Gamma^d$ of the equation 
$$
c_1X_1+\cdots + c_dX_d =1
$$
is a $p$-automatic subset of $\Gamma^d$ that can be effectively determined. 
\end{thm}

\begin{proof} 
Let $$S:= \left\{ (x_1,\ldots,x_d)\in \Gamma^{d} \mid c_1x_1+\cdots + c_dx_d =1\right\}\, .$$ 
Our aim is to prove that $S$ is $p$-automatic and can be effectively determined.  

\medskip

We first fix some notation. Let $g_1,\ldots,g_m$ be a set of generators of $\Gamma$ and  
let us consider a surjective group homomorphism $\Phi:\mathbb{Z}^{m}\rightarrow \Gamma$.  
This allows us to define a surjective group homomorphism 
$\tilde{\Phi}:(\mathbb{Z}^{m})^d\rightarrow \Gamma^d$ by 
$\tilde{\Phi}({\bf x}_1,\ldots,{\bf x}_d)= (\Phi({\bf x}_1),\ldots, \Phi({\bf x}_d))$. 
By Proposition \ref{prop:automorph}, it is enough to show that $\tilde{\Phi}^{-1}(S)$ is a 
$p$-automatic subset of $(\mathbb{Z}^{m})^d\simeq \mathbb Z^{m\times d}$. 
Let ${\mathcal E} \ := \ \{\pm 1\}^m.$   
Given ${\bf n}:=(n_1,\ldots ,n_m)\in \mathbb N^m$ and ${\bf a}:=(a_1, \ldots, a_m)\in \mathcal E$, 
we let ${\bf a}\cdot {\bf n}:= (a_1n_1,\ldots,a_mn_m)$ denote the ordinary coordinate-wise multiplication.  
Given $A\subseteq \mathbb N^m$, we also set  
${\bf a}\cdot A:= \left\{ {\bf a}\cdot {\bf n} \mid {\bf n} \in A\right\}$.
For every ${\bf a}:=({\bf a}_1, \ldots,{\bf a}_d)\in {\mathcal E}^d$, we set 
$$
S_{\bf a} := \left\{({\bf n}_1,\ldots,{\bf n}_d)\in \mathbb N^{m\times d} \mid c_1\Phi({\bf a}_1 \cdot {\bf n}_1)+
\cdots + c_d \Phi({\bf a}_d \cdot {\bf n}_d) =1\right\}\, .
$$ 
Thus 
\begin{equation}\label{eq: phi}
\tilde{\Phi}^{-1}(S)= \bigcup_{{\bf a}\in {\mathcal E}^d} {\bf a}\cdot S_{\bf a} \, .
\end{equation}
Note that by Proposition \ref{prop:autoZequiv}, $S$ is 
$p$-automatic subset of $\Gamma^d$ if and only if $S_{\bf a}$ is a $p$-automatic subset 
of $\mathbb N^d$ for every ${\bf a}\in {\mathcal E}^d$. 

\medskip

We let $t_{i,j}$ be indeterminates for $1\le i\le d$ and $1\le j\le m$.  
We define ${\bf t}_i=(t_{i,1},\ldots ,t_{i,m})$ for $1\le i\le d$.    
Given ${\bf n}\in \mathbb{N}^m$ and $i\in \{1,2,\ldots ,d\}$, 
we define ${\bf t}_i^{\bf n}$ to be the product
$t_{i,1}^{n_1}\cdots t_{i,m}^{n_m}$.   
Given  ${\bf a}:=({\bf a}_1,\ldots ,{\bf a}_d)\in {\mathcal E}^d$, 
we define the function 
$$f_{\bf a}({\bf t}_1,\ldots ,{\bf t}_d):= \sum_{{\bf n}_1,\ldots ,{\bf n}_d\in \mathbb{N}^{m}}
\left( -1+\sum_{i=1}^d c_i \Phi({\bf a}_i\cdot {\bf n}_i)\right){\bf t}_1^{{\bf n}_1}\cdots 
{\bf t}_d^{{\bf n}_d}.$$
This definition ensures that 
\begin{equation}\label{eq: sa}
S_{\bf a} = {\mathcal Z}(f_{\bf a})  \, .
\end{equation}
For every $i\in \{1,2,\ldots ,d\}$, we also set ${\bf n}_i=(n_{i,1},\ldots, n_{i,m})$ and 
${\bf a}_i:=(a_{i,1},\ldots,a_{i,m})$.  
Let ${\bf e}_j=(0,0,\ldots ,0,1,0\ldots ,0)\in \mathbb{Z}^m$ denote the element 
whose $j$th coordinate is $1$ and whose other coordinates are $0$.  
Then for every $i\in \{1,\ldots ,d\}$, we have
\begin{eqnarray*}  \sum_{{\bf n}_i\in \mathbb{N}^{m}} c_i\Phi({\bf a}_i\cdot {\bf n}_i){\bf t}_i^{{\bf n}_i}
&=& \sum_{n_{i,1}=0}^{\infty}\cdots \sum_{n_{i,m}=0}^{\infty} \prod_{j=1}^m \Phi({\bf e}_j)^{a_{i,j}n_{i,j}} t_{i,j}^{n_{i,j}}\\
&=& \prod_{j=1}^m (1-\Phi({\bf e}_j)^{a_{i,j}} t_{i,j})^{-1}
\end{eqnarray*}
is a rational function.  Hence
\begin{eqnarray*} f_{\bf a}({\bf t}_1,\ldots ,{\bf t}_d) &=& \prod_{i=1}^d\prod_{j=1}^m (1-t_{i,j})^{-1}\left( -1 + \sum_{i=1}^d c_i \prod_{j=1}^m \frac{(1-t_{i,j})}{(1-\Phi({\bf e}_j)^{a_{i,j}} t_{i,j})}\right)
\end{eqnarray*} is a rational function for each ${\bf a}\in {\mathcal E}^d$.  
Since we get an explicit  expression for the function $f_{\bf a}$ 
(assuming that we explicitly know a set of generators $g_1,\ldots, g_m$ of $\Gamma$), 
we infer from Theorem \ref{thm: eff} that the 
set ${\mathcal Z}(f_{\bf a})$ is a $p$-automatic subset of $\mathbb N^d$ 
which  can be effectively determined. By (\ref{eq: phi}) and (\ref{eq: sa}), 
this ends the proof. 
\end{proof}

\section{An effective result related to the Mordell--Lang theorem }\label{MLT}

The expression ``Mordell--Lang theorem" or ``Mordell--Lang conjecture" 
serves as a generic appellation which denotes results describing the structure of 
intersections of the form 
$$X\cap \Gamma\, ,$$
where $X$ is a subvariety (Zariski closed subset) of 
a (affine, abelian, or semi-abelian) variety $A$ and $\Gamma$ is a finitely generated subgroup 
(or even a subgroup of finite rank) of $A$.  
The case where the variety $A$ is defined over a field of characteristic $0$ has many interesting Diophantine consequences, 
including the famous Faltings' theorem \cite{Fal}.  

On the other hand, simple examples constructed using the Frobenius endomorphism 
(as in Section \ref{sunit}) show that such intersections may behave differently 
when the variety $A$ is defined over a field of positive characteristic.  
Hrushovski \cite{Hru} proved a relative version of the Mordell--Lang conjecture 
for semi-abelian varieties defined over a field $K$ of positive characteristic. 
His approach, which makes use of model theory, has then been pursued by several authors  
(see for instance \cite{MS} and \cite{Ghioca}). 

\medskip

All general results known up to now in this direction seem to be ineffective.  
The aim of this section is to prove the two following effective statements. 
We recall that the notion of an automatic subset of a 
finitely generated abelian group is given in Section \ref{Salon} (see Definition \ref{def:autofgg} and 
Proposition \ref{prop:fggcarac}).

\begin{thm}\label{thm: MT}
Let $K$ be a field of characterisitc $p$ and let $d$ be a positive integer. 
Let $X$ be a Zariski closed subset of ${\rm GL}_d(K)$ and 
$\Gamma$ a finitely generated abelian subgroup of ${\rm GL}_d(K)$. 
Then the set $X\cap \Gamma$ is a $p$-automatic subset of $\Gamma$ 
that can be effectively determined. 
\end{thm}

Note more generally that, given positive integers $d_1,\ldots,d_n$, 
the same result holds for Zariski closed subsets of 
$\prod_{i=1}^n {\rm GL}_{d_i}(K)$.  
Indeed, we have a natural embedding $\i$ of $\prod_{i=1}^n {\rm GL}_{d_i}(K)$ 
as a Zariski closed subset of ${\rm GL}_{d_1+\cdots +d_n}(K)$, where $\i$ sends an 
$n$-tuple of  invertible matrices in which the $i$th matrix has size $d_i\times d_i$ 
to the block diagonal matrix with $n$ blocks whose $i$th block 
is the $i$th coordinate of our $n$-tuple.  Indeed, under this identification, 
$\prod_{i=1}^n {\rm GL}_{d_i}(K)$ 
is the zero set of the linear polynomials $x_{i,j}$ for which 
$i$ and $j$ have the property that there does not exist a positive integer 
$k$, $k\le n$, such that $$d_0+\cdots + d_{k-1} < i,j\le d_1+\cdots + d_k\, ,$$
where we take $d_0$ to be zero. 
Given a Zariski closed subset $X$ of $\prod_{i=1}^n {\rm GL}_{d_i}(K)$, 
we thus may regard $X$ as a Zariski closed subset of ${\rm GL}_{d_1+\cdots +d_n}(K)$.  
We note that the additive torus embeds in ${\rm GL}_2(K)$ by identifying the torus with unipotent 
upper-triangular matrices.  Moreover, this is easily seen to be a Zariski closed subset of 
${\rm GL}_2(K)$.  Applying these remarks with $d_1,\ldots ,d_n\in \{1,2\}$, 
we deduce the following corollary.

\begin{cor}\label{thm: tori}
Let $K$ be a field of characterisitc $p$ and let $s$ and $t$ be nonnegative integers. 
Let $X$ be a subvariety of ${\rm G}_a^s(K)\times {\rm G}_m^t(K)$ and 
$\Gamma$ a finitely generated subgroup of ${\rm G}_a^s(K)\times {\rm G}_m^t(K)$. 
Then the set $X\cap \Gamma$ is a $p$-automatic subset of $\Gamma$ 
that can be effectively determined. 
\end{cor}

We note that one can actually obtain an ineffective version of Theorem \ref{thm: MT} 
from Corollary \ref{thm: tori}. In fact, one only needs to consider multiplicative tori.  
To see this, we observe that if $\Gamma$ is a finitely generated abelian subgroup of 
${\rm GL}_d(K)$, then by considering Jordan forms, there is some natural number $n$ 
such that $g^{p^n}$ is diagonalizable for every $g\in \Gamma$.  As commuting diagonalizable 
operators are simultaneously diagonalizable, we may replace $K$ by a finite extension $K'$ that 
contains the eigenvalues of $g^{p^n}$ as $g$ ranges over a generating set, and assume that 
$\Gamma^{p^n}$ is a subgroup of $T\cong {\rm G}_m^d(K')$, the invertible diagonal matrices in 
${\rm GL}_d(K')$.  As $X\cap T$ is Zariski closed in $T$ and
$X\cap \Gamma^{p^n}=(X\cap T)\cap \Gamma^{p^n}$, Corollary \ref{thm: tori} applies and 
so $\Gamma^{p^n}\cap X$ is $p$-automatic.  By applying a suitable translate, it follows that the 
intersection of $X$ with each coset of $\Gamma/\Gamma^{p^n}$ is $p$-automatic.  
As there are only finitely many cosets, using basic properties of automaticity, we deduce 
that $\Gamma\cap X$ is $p$-automatic.  

It is however less clear whether an effective version of Theorem \ref{thm: MT} can be obtained from 
Corollary \ref{thm: tori}. Indeed, to determine the intersection using the method described above in practice, 
one must be able to explicitly find eigenvectors in order to diagonalize elements of $\Gamma^{p^n}$.  A necessary step in doing this is to find roots of characteristic polynomials in the algebraic closure of $K$,  
which seems uneasy to be done explicitly in general.

It is also natural to ask whether a similar version of Theorem \ref{thm: MT} might hold for abelian varieties.  
We believe this to be the case, but it is not clear whether the result follows from our approach: 
if $P$ is a point on an abelian variety $X$ over a field of positive characteristic then the points $n\cdot P$ 
do not appear, in general, to be sufficiently well-behaved to allow one to associate an algebraic generating function, which is necessary to apply our methods.

\begin{proof}[Proof of Theorem \ref{thm: MT}]  
We first make a few reductions.  We let $\Phi: {\rm GL}_d(K)\to \mathbb{A}^{d^2}(K)$ 
be the injective morphism whose image, $Y$, consists of all points at which the determinant 
does not vanish. Note that the affine variety
${\rm GL}_d(K)$ is a Zariski open subset of $\mathbb{A}^{d^2}(K)$ and that the 
Zariski closed subsets of ${\rm GL}_d(K)$ are precisely those obtained by intersecting Zariski 
closed subsets of $\mathbb{A}^{d^2}(K)$ with ${\rm GL}_d(K)$.  
By the Hilbert Basis Theorem, a Zariski closed subset of $\mathbb{A}^{d^2}(K)$ 
is given by the vanishing set of a finite set of polynomials.  
Thus there are polynomials $P_1,\ldots,P_r\in K[x_{1,1},\ldots ,x_{d,d}]$
 such that for $M\in {\rm GL}_d(K)$, 
$$M\in X \iff P_1(\Phi(M))=\cdots = P_r(\Phi(M))=0\, .$$
 It is then enough to consider the case that $\Phi(X)=Z(P)\cap Y$, 
where $P$ is a single polynomial in the indeterminates $x_{i,j}$ with $1\le i,j\le d$ and 
$Z(P)$ denotes the set of zeros of $P$. 
 
Let  $\Gamma$ be a finitely generated abelian subgroup of ${\rm GL}_d(K)$
and let $X$ be a Zariski closed subset of ${\rm GL}_d(K)$ such that 
$\Phi(X)=Z(P)\cap Y$, where $P\in K[x_{1,1},\ldots ,x_{d,d}]$. Our aim is to prove that 
$X\cap\Gamma$ is a $p$-automatic subset of $\Gamma$. 
Let $C_1,\ldots ,C_m\in {\rm GL}_d(K)$ be generators of $\Gamma$ 
and suppose that $\Psi:\mathbb{Z}^{m}\rightarrow \Gamma$ is the surjective group 
homomorphism defined by $\Psi(e_i)=C_i$ for $1\le i\le m$, where $e_i$ stands for 
the vector whose $i$th coordinate is $1$ and all other coordinates are $0$. 
We let ${\bf n}$ denote an $m$-tuple $(n_1,\ldots ,n_m)\in \mathbb{N}^m$.  
By Proposition \ref{prop:automorph}, $X\cap \Gamma$ is $p$-automatic if 
$$
S:= \left\{{\bf n}\in \mathbb{Z}^m \mid P(\Phi\circ \Psi({\bf n}) )=0 \right\}
$$
is a $p$-automatic subset of $\mathbb Z^m$. 
Let ${\mathcal E} \ := \ \{\pm 1\}^m.$   
Given ${\bf n}:=(n_1,\ldots ,n_m)\in \mathbb N^m$ and ${\bf a}:=(a_1, \ldots, a_m)\in \mathcal E$, 
we denote by ${\bf a}\cdot {\bf n}:= (a_1n_1,\ldots,a_mn_m)$ the ordinary coordinate-wise multiplication.  
Given $A\subseteq \mathbb N^m$, we also set  
${\bf a}\cdot A:= \left\{ {\bf a}\cdot {\bf n} \mid {\bf n} \in A\right\}$.
For every ${\bf a}\in {\mathcal E}$, we set 
$$
S_{\bf a} := \left\{{\bf n} \in \mathbb N^m \mid P(\Phi\circ \Psi({\bf a}\cdot {\bf n}) )=0 \right\}\, .
$$  
Note that by Proposition \ref{prop:autoZequiv}, $S$ is 
a $p$-automatic subset of $\mathbb Z^m$ if and only if $S_{\bf a}$ is a $p$-automatic subset 
of $\mathbb N^m$ for every ${\bf a}\in {\mathcal E}$.

To see this, let $t_{j}$ be indeterminates for $1\le j\le m$.  
Given ${\bf n}\in \mathbb{N}^m$, we define ${\bf t}^{\bf n}$ to be the product
$t_{1}^{n_1}\cdots t_{m}^{n_m}$.  
Let ${\bf a}=(a_1,\ldots ,a_m)\in {\mathcal E}$.  We set
\begin{equation*} 
f_{\bf a}({\bf t}) := \sum_{{\bf n}\in \mathbb{N}^{m}} \Psi({\bf a \cdot n}) {\bf t}^{\bf n} \in {\rm GL}_d(K)[[{\bf t}]] \, .
\end{equation*}
We claim that for $1\le i,j\le d$, the $(i,j)$ entry of $\Psi({\bf a \cdot n}) {\bf t}^{\bf n}$ is a rational 
function in ${\bf t}$.  To see this, first note that since $C_1,\ldots ,C_m$ commute, we have
\begin{eqnarray*}
f_{\bf a}({\bf t})&=& \sum_{(n_1,\ldots,n_m)\in\mathbb N^m}\Psi(C_1^{a_1n_1},\ldots,C_m^{a_mn_m}) t_1^{n_1}\cdots t_m^{n_m}\\
&=& \prod_{i=1}^m \, \sum_{n_i\in\mathbb N} C_i^{a_i n_i} t_i^{n_i}.
\end{eqnarray*}
On the other hand, for every $i\in\{1,\ldots,m\}$, the sum
$$ \sum_{n_i\in \mathbb N} C_i^{a_i n_i} t_i^{n_i}$$ 
is a $d\times d$ matrix whose entries are rational functions 
that belong to $K(t_i)$. This follows for instance from Proposition 1.1 in \cite{Han86}.  
Since rational functions are closed under Hadamard product and taking linear combinations, 
we obtain that $f_{\bf a}({\bf t})$ is a $d\times d$ matrix whose entries are all 
multivariate rational functions in ${\bf t}$.  
For all $1\le i,j\le d$, let us denote by $f_{i,j,{\bf a}} ({\bf t})$ the $(i,j)$ entry of $f_{\bf a}({\bf t})$. 
Note that the power series 
\[ \tilde{f}_{\bf a}({\bf t}) \ := \   \sum_{{\bf n}\in \mathbb{N}^{m}} 
P( \Phi\circ \Psi({\bf a \cdot n}))  {\bf t}^{\bf n}\]
can be obtained by taking Hadamard product and linear combinations of the rational functions 
$f_{i,j,{\bf a}} ({\bf t})$. We thus deduce that $ \tilde{f}_{\bf a}({\bf t})$ 
belongs to the field of multivariate rational functions $K({\bf t})$.  
On the other hand, the definition of $ \tilde{f}_{\bf a}$ implies that 
$$
S_{\bf a} = {\mathcal Z}( \tilde{f}_{\bf a}) \,.
$$
By Theorem \ref{thm: eff}, we have that the set $S_{\bf a}$ is a $p$-automatic set 
that can be effectively determined. Since this holds true for evey ${\bf a}\in {\mathcal E}$,  
this ends the proof.
\end{proof}
\section{Background from automata theory}
\label{Salon}

We start this section with  few examples of automatic sequences and automatic subsets 
of the natural numbers, as well as a useful chatacterization of them (Theorem \ref{AB:theorem:eilenberg}).  
Then we describe Salon's \cite{Salon1987} extension of the notion of automatic 
sets to subsets of $\mathbb{N}^d$ and show how to genralize it to subsests of $\mathbb Z^d$. 
Finally, we introduce a natural notion of automaticity 
for subsets of arbitrary finitely generated abelian groups. It seems that the latter notion has not been 
considered before and that it could be of independent interest. 

\medskip

Let $k\ge 2$ be a natural number.  We let $\Sigma_k$ denote the alphabet  
$\left\{0,1,\ldots,k-1\right\}$. 

\subsection{Automatic sequences and one-dimensional automatic sets} 

For reader's convenience we choose to recall here the definitions of a $k$-automatic sequence and a $k$-automatic subset
of the natural numbers. 

A $k$-automaton\index{$k$-automaton} is a $6$-tuple 
\begin{displaymath}
{\mathcal A} = \left(Q,\Sigma_k,\delta,q_0,\Delta,\tau\right) ,
\end{displaymath}
where $Q$ is a finite set of states, 
$\delta:Q\times \Sigma_k\rightarrow Q$ is the transition function, $q_0$ is 
the initial state, $\Delta$ is the output alphabet and $\tau : Q\rightarrow 
\Delta$ is the output function. For a state $q$ in $Q$ and for a finite 
word $w=w_1 w_2 \cdots w_n$ on the alphabet $\Sigma_k$, 
we define $\delta(q,w)$ recursively by 
$\delta(q,w)=\delta(\delta(q,w_1w_2\cdots w_{n-1}),w_n)$. 
Let $n\geq 0$ be an integer and let 
$w_r w_{r-1}\cdots w_1 w_0$ in $\left(\Sigma_k\right)^{r+1}$ 
be the base-$k$ expansion 
of $n$. Thus $n=\sum_{i=0}^r w_i k^{i} :=[w_rw_{r-1}\cdots w_0]_k$. We denote by $w(n)$ the 
word $w_0 w_1 \cdots w_r$. 
A sequence $(a_n)_{n\geq 0}$ is 
said to be $k$-automatic if there exists a $k$-automaton ${\mathcal A}$ such that 
$a_n=\tau(\delta(q_0,w(n)))$ for all $n\geq 0$. 
A set ${\mathcal E}\subset \mathbb N$\index{automatic set} is said to be recognizable 
by a finite $k$-automaton,  
or for short $k$-automatic, if the characteristic sequence of ${\mathcal E}$, defined by 
$a_n=1$ if $n\in {\mathcal E}$ and $a_n=0$ otherwise, is a $k$-automatic sequence. 

\begin{exam}{\em 
The Thue--Morse sequence\index{Thue--Morse sequence} $t:=(t_n)_{n\geq 0}$ 
is probably the famous example of automatic sequences. It is defined as follows: 
$t_n=0$ if the sum of the binary digits of $n$ is even, and $t_n=1$ otherwise.  
The Thue--Morse sequence can be generated by the following 
finite $2$-automaton:    
$
{\mathcal A}=\left ( \{A, B\}, \{0, 1\}, \delta, A, \{0, 1\}, \tau \right)$, 
where
$
\delta(A, 0) = \delta (B, 1) = A$, $\delta(A, 1) = \delta (B, 0) = B$,  
 $\tau (A) = 0$  and  $\tau (B) = 1$. 
\begin{figure}[htbp]
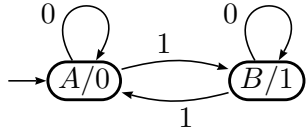

\centering
\VCDraw{%
        \begin{VCPicture}{(0,-1)(4,2)}
 \StateVar[A/0]{(0,0)}{a}  \StateVar[B/1]{(4,0)}{b}
\Initial[w]{a}
\LoopN{a}{0}
\LoopN{b}{0}
\ArcL{a}{b}{1}
\ArcL{b}{a}{1}
\end{VCPicture}
}
\caption{A $2$-automaton generating Thue--Morse sequence.}
  \label{AB:figure:thue}
\end{figure}}
\end{exam}

\begin{exam} {\em 
The simplest automatic sets are arithmetic progressions. 
\begin{figure}[h]
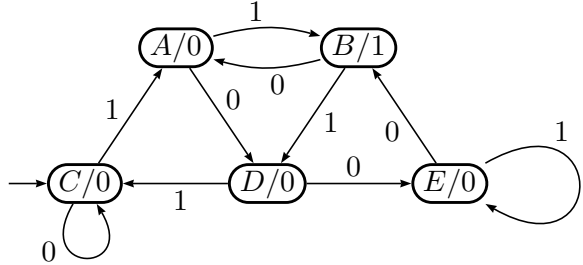

\centering 
\VCDraw{  \begin{VCPicture}{(0,-4.5)(4,1)}
    \StateVar[A/0]{(0,0)}{A} \StateVar[B/1]{(4,0)}{B}   \StateVar[C/0]{(-2,-3)}{C}    \StateVar[D/0]{(2,-3)}{D} 
     \StateVar[E/0]{(6,-3)}{E}   
    \Initial{C} 
     \EdgeL{C}{A}{1}     \EdgeL{A}{D}{0}  \EdgeL{B}{D}{1} 
     \LoopS{C}{0}    \LoopE{E}{1} \ArcL{A}{B}{1} \EdgeL{E}{B}{0}
\ArcL{B}{A}{0}  \EdgeL{D}{C}{1}  \EdgeL{D}{E}{0}
  \end{VCPicture}}
  \caption{A $2$-automaton recognizing the arithmetic progression $5\mathbb N+3$.}
  \label{AB:figure:ap}
\end{figure}}
\end{exam}

\begin{exam}{\em 
The set $\{1,2,4,8,16,\ldots\}$ formed by the powers of $2$ is also a typical example 
of a $2$-automatic set. 
\begin{figure}[h]
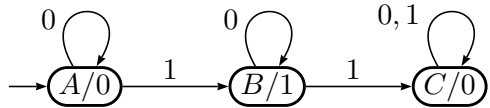

\centering 
\VCDraw{  \begin{VCPicture}{(0,-1)(8,2)}
    \StateVar[A/0]{(0,0)}{A} \StateVar[B/1]{(4,0)}{B}   \StateVar[C/0]{(8,0)}{C}  
    \Initial{A} 
     \EdgeL{A}{B}{1}  \EdgeL{B}{C}{1} \LoopN{A}{0} 
    \LoopN{B}{0}  \LoopN{C}{0,1}
  \end{VCPicture}}
  \caption{A $2$-automaton recognizing the powers of $2$.}
  \label{AB:figure:2n}
\end{figure}
}
\end{exam}

\begin{exam}{\em 
In the same spirit, the set formed by taking all integers 
that can be expressed as the sum of at most 
two powers of $3$ is $3$-automatic. 
\begin{figure}[h]
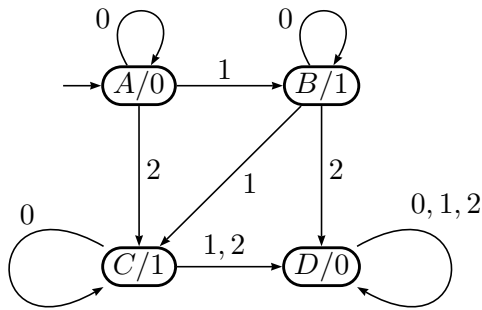

\centering 
\VCDraw{  \begin{VCPicture}{(0,-5)(4,2)}
    \StateVar[A/0]{(0,0)}{A} \StateVar[B/1]{(4,0)}{B}   \StateVar[C/1]{(0,-4)}{C}    \StateVar[D/0]{(4,-4)}{D}  
    \Initial{A} 
     \EdgeL{A}{B}{1}  \EdgeL{B}{C}{1}   \EdgeL{A}{C}{2}   \LoopN{A}{0} 
    \LoopN{B}{0}  \LoopW{C}{0}   \LoopE{D}{0,1,2}  \EdgeL{B}{D}{2} \EdgeL{C}{D}{1,2}
  \end{VCPicture}}
  \caption{A $3$-automaton recognizing those integers that are the sum of at most two powers of $3$.}
  \label{AB:figure:3n3m}
\end{figure}}
\end{exam}

\medskip

 There are also much stranger automatic sets.  The fact that 
the class of $k$-automatic sets is closed under various natural operations such 
as intersection, union and complement, can actually be used to easily construct rather sophisticated 
automatic sets.  
For instance, the set of integers whose binary expansion has an odd number of digits, does not 
contain three consecutive $1$'s, 
and contains an even number of two consecutive $0$'s is a $2$-automatic set.

\medskip

An important notion in the study of $k$-automatic sequences is the notion of $k$-kernel. 

\begin{defn}{\em The $k$-kernel of a sequence $a=(a_n)_{n\geq 0}$ is defined as the set 
\begin{displaymath}
 \left\{(a_{k^in+j})_{n\geq 0}  \mid
i\geq 0, \,0\leq j<k^i \right\} \, .
\end{displaymath}}
\end{defn}
 
 \begin{exam}\label{AB:example:kernel}{\em 
The $2$-kernel of the Thue--Morse sequence $t$ has only two elements 
$t$ and the sequence $\overline{t}$ obtained by exchanging the symbols $0$ and $1$ in $t$.}
\end{exam}

This notion gives rise to a useful characterization of $k$-automatic sequences 
which was first proved by Eilenberg in \cite{Eilenberg}.

\begin{thm}[Eilenberg]\label{AB:theorem:eilenberg} 
A sequence is $k$-automatic if and only if its $k$-kernel is finite.
\end{thm}

\subsection{Automatic subsets of $\mathbb N^d$ and multidimensional 
automatic sequences}

Salon \cite{Salon1987} extended the notion of automatic sets to include subsets of 
$\mathbb{N}^d$, where $d\ge 1$.  To do this, we consider an automaton   
\begin{displaymath}
{\mathcal A} = \left(Q,\Sigma_k^d,\delta,q_0,\Delta,\tau\right) \, ,
\end{displaymath}
where $Q$ is a finite set of states, 
$\delta:Q\times \Sigma_k^d\rightarrow Q$ is the transition function, $q_0$ is 
the initial state, $\Delta$ is the output alphabet and $\tau : Q\rightarrow 
\Delta$ is the output function. Just as in the one-dimensional case, for a state $q$ in $Q$ and for a finite 
word $w=w_1 w_2 \cdots w_n$ on the alphabet $\Sigma_k^d$, 
we recursively define $\delta(q,w)$ by 
$\delta(q,w)=\delta(\delta(q,w_1w_2\cdots w_{n-1}),w_n)$.  
We call such an automaton a $d$-dimensional $k$-automaton.

We identify $\left(\Sigma_k^d\right)^*$ with the subset of $\left( \Sigma_k^*\right)^d$ consisting 
of all $d$-tuples $(u_1,\ldots ,u_d)$ such that $u_1,\ldots ,u_d$ all have the same length. 
Each nonnegative integer $n$ can be written uniquely as 
\begin{displaymath}
n \ = \ \sum_{j=0}^{\infty} e_j(n) k^j \, ,
\end{displaymath} 
in which $e_j(n)\in \{0,\ldots ,k-1\}$ and $e_j(n)=0$ for all sufficiently large $j$.   
Given a nonzero $d$-tuple of nonnegative integers $(n_1,\ldots ,n_d)$, we set 
\begin{displaymath}
h:= \max \{ j \geq 0~\mid~ {\rm there ~exists ~some}~i~,\; 1\leq i \leq d,~{\rm such ~that ~}e_j(n_i)\neq 0\}\, .
\end{displaymath}
Furthermore, if $(n_1,\ldots,n_d)=(0,\ldots,0)$, we set $:h=0$. 

  We can then produce an element 
 \begin{displaymath}
 w_k(n_1,\ldots ,n_d):=(w_1,\ldots ,w_d)\in \left(\Sigma_k^d\right)^*
 \end{displaymath}
  corresponding to $(n_1,\ldots ,n_d)$ by defining
\begin{displaymath}
w_i := e_h(n_i)e_{h-1}(n_i)\cdots e_0(n_i)\, .
\end{displaymath}
  In other words, we are taking the base-$k$ expansions of $n_1,\ldots ,n_r$ and 
  then ``padding" the expansions of each $n_i$ at the beginning with $0$'s if necessary 
  to ensure that each expansion has the same length. 
  
  \begin{exam} If $d=3$ and $k=2$, then we have $w_2(3,5,0)=(011, 101,000)$.  
\end{exam}

\begin{defn}{\em A map $f:\mathbb{N}^d\rightarrow \Delta$ is $k$-automatic 
if there is a $d$-dimensional $k$-automaton\index{multidimensional automaton} 
${\mathcal A} = \left(Q,\Sigma_k^d,\delta,q_0,\Delta,\tau\right)$ such that 
\begin{displaymath}
f(n_1,\ldots ,n_d) = \tau(\delta(q_0,w_d(n_1,\ldots ,n_d)))\, .
\end{displaymath}
Similarly,  a  subset $S$ of $\mathbb{N}^d$ is $k$-automatic if 
its characteristic function, $f:\mathbb{N}^d\to \{0,1\}$, 
defined by $f(n_1,\ldots ,n_d)=1$ if $(n_1,\ldots ,n_d)\in S$; and $f(n_1,\ldots ,n_d)=0$, 
otherwise, is $k$-automatic.}
\end{defn}

\begin{exam}\label{exam: 2dim}{\em 
Let $f:\mathbb{N}^2\rightarrow \{0,1\}$ be defined by $f(n,m)=1$ if the sum 
of the binary digits of $n$ added to the sum of the binary digits of $m$ is even, 
and $f(n,m)=0$ otherwise.  Then $f(m,n)$ is a $2$-automatic map. 
One can check that $f$ can be generated by the following 
$2$-dimensional $2$-automaton:    
$
{\mathcal A}=\left ( \{A, B\}, \{0, 1\}^2, \delta, A, \{0, 1\}, \tau \right)$,
where
$\delta(A, (0,0)) = \delta (A, (1,1)) = \delta(B,(1,0)) = \delta(B,(0,1)) = A$, 
$\delta(A, (1,0)) = \delta(A,(0,1)) = \delta (B, (0,0)) = \delta (B, (1,1))= B$,  
 $\tau (A) = 0$ and  $\tau (B) = 1$.}
\end{exam}

\begin{figure}[h]
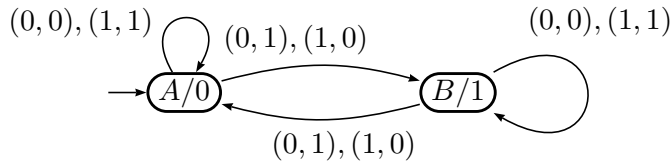

\centering 
\VCDraw{%
        \begin{VCPicture}{(0,-2)(4,2)}
 \StateVar[A/0]{(0,0)}{a}  \StateVar[B/1]{(6,0)}{b}
\Initial[w]{a}
\LoopN{a}{(0,0),(1,1)}
\LoopE{b}{(0,0),(1,1)}
\ArcL{a}{b}{(0,1),(1,0)}
\ArcL{b}{a}{(0,1),(1,0)}
\end{VCPicture}
}
\caption{A $2$-dimensional $2$-automaton generating the map $f$ defined in Example \ref{exam: 2dim}.}
  \label{AB:figure:thue2dim}
\end{figure}

Just as $k$-automatic sequences can be characterized by the finiteness of the $k$-kernel, 
multidimensional $k$-automatic sequences have a similar characterization.

\begin{defn}  {\em Let $d$ be a positive integer and let $\Delta$ be a finite set.  
We define the $k$-\emph{kernel}\index{kernel} of a map $f:\mathbb{N}^d\rightarrow \Delta$ 
to be the collection of all maps of the form
\begin{displaymath}
g(n_1,\ldots ,n_d):=f(k^a n_1+b_1,\ldots ,k^a n_d+b_d)
\end{displaymath}
 where $a\ge 0$ and $0\le b_1,\ldots ,b_d<k^a$.  }
 \label{defn: kerd}
\end{defn}

\begin{exam}{\em 
The $2$-kernel of the map $f:\mathbb{N}^2\rightarrow \{0,1\}$ defined in Example \ref{exam: 2dim} 
 consists of the $2$ maps $f_1(m,n):=f(m,n)$, $f_2(m,n)=f(2m+1,2n)$.  }
\end{exam}

Just as Eilenberg \cite{Eilenberg} showed that being $k$-automatic is equivalent to 
having a finite $k$-kernel for $k$-automatic sequences, Salon \cite[Theorem 1]{Salon1987} 
observed that a similar characterization of multidimensional $k$-automatic maps holds.

\begin{thm}[Salon]\label{thm: Salon} 
Let $d$ be a positive integer and let $\Delta$ be a finite set.
A map $f:\mathbb{N}^d\to \Delta$ is $k$-automatic if and only if its $k$-kernel is finite.
\end{thm}

\subsection{Automatic subsets of $\mathbb Z^d$}

We show now how to naturally extend Salon's construction to $k$-automatic 
subsets of $\mathbb Z^d$ by simply adding sympols $+$ and $-$ to our alphabet $\Sigma_k$.  

\medskip
 
Given a natural number $n$, we let $[n]_k$ 
denote the base-$k$ expansion of $n$. 
 We set \begin{equation*}
\Sigma'_k \: = \
\{0,1,\ldots ,k-1,-,+ \}\end{equation*}
 and we let
$\mathcal{L}_k$ denote the language over the alphabet $\Sigma'_k$ consisting of the empty word 
and all words over $\Sigma'_k$ whose length is at least $2$ such that the initial letter is either $+$ or $-$, 
the remaining letters are all in $\Sigma_k$, and the last letter is not equal to zero.
This is easily seen to be a regular language.
 
There is a bijection $[\ \cdot \ ]_k:\mathcal{L}(k)\rightarrow \mathbb{Z}$ 
in which the empty word is sent to zero, 
\begin{equation*}
+s_0\cdots s_{n}\ \in \mathcal{L}(k) \ \mapsto \
\sum_{j=0}^{n} s_j k^{j}
\end{equation*}
and
\begin{equation*}
-s_0\cdots s_{n}\ \in \mathcal{L}(k) \ \mapsto \
-\sum_{j=0}^{n} s_j k^{j}\, ,
\end{equation*}
where $s_0,\ldots ,s_n\in \{0,1,\ldots ,k-1\}$.

\begin{defn} {\em
We say that a subset $S$ of $\mathbb{Z}$ 
is $k$-\emph{automatic} if 
there is a finite-state automaton that takes words
over $\Sigma'_k$ as input, and has the property that a word $W\in \mathcal{L}_k$ is accepted by
the automaton if and only if 
$[W]_k\in S$. }
\end{defn}

More generally, we can define automatic subsets of $\mathbb{Z}^d$, mimicking the 
construction of Salon \cite{Salon1987}.  For a natural number $d\ge 1$, we create the 
alphabet $\Sigma_k'(d)$ 
to be the alphabet $\left(\Sigma_k' \right)^d$ consisting of all $d$-tuples of elements of $\Sigma_k'$.  
With this in mind, we construct a regular language $\mathcal{L}_k(d)
\subseteq \left(\Sigma_k'(d)\right)^*$ as follows.  Given a nonzero integer $n$, we can write it 
uniquely as 
$$n \ = \ \epsilon \sum_{j=0}^{\infty} e_j(n) k^j\, ,$$ 
in which $\epsilon\in  \{\pm 1\}$, $e_j(n)\in \{0,\ldots ,k-1\}$ and there is some natural number $N$, 
depending on $n$, such that $e_j(n)=0$ whenever $j>N$.   We also set 
$$
0  \ = \ + \sum_{j=0}^{\infty} e_j(0) k^j\, ,
$$
where $e_j(0)=0$ for all $j\geq 0$. 
Given  a nonzero $d$-tuple of integers $(n_1,\ldots ,n_d)$, we set
$$h:= \max\{ j~\mid~ {\rm there ~exists ~some}~i~{\rm such ~that ~}e_j(n_i)\neq 0\}\, .$$
If $(n_1,\ldots,n_d)=(0,\ldots,0)$, we set $:h=0$. 
 
 We can then produce an element $$w_k(n_1,\ldots ,n_d):=(w_1,\ldots ,w_d)\in
  \left(\Sigma_k'(d)\right)^*$$  corresponding to $(n_1,\ldots ,n_d)$ by defining
$$w_i := \epsilon_i e_h(n_i)e_{h-1}(n_i)\cdots e_0(n_i)\, ,$$ 
where $\epsilon_i$ is $+$ if $n_i$ is nonnegative and is $-$ if $n_i<0$.  
In other words, we are taking the base $k$-expansions of $n_1,\ldots ,n_d$ 
and then ``padding'' the expansions of each $n_i$ at the beginning to ensure that each 
expansion has the same length.  

\begin{exam} If $d=3$ and $k=2$, then we have $w_3(14,-3,0)=(+1110, -0011,+0000)$.  
\end{exam}

We then take $\mathcal{L}_k(d)$ to be the collection of words of the form
$$w_k(n_1,\ldots ,n_d)$$ where $(n_1,\ldots ,n_d)\in \mathbb Z^d$.   Then there is 
an obvious way to extend the map $[\cdot ]_k$ to a bijection $[\, \cdot \, ]_k :\mathcal{L}_k(d)\rightarrow \mathbb{Z}^d$; namely, 
$$[w_k(n_1,\ldots ,n_d)]_k \ := \   (n_1,\ldots ,n_d)\, .$$
We also denote by $[\, \cdot \, ]_k^{-1}$ the reciprocal map. 

We can now define the notion of a $k$-automatic function from $\mathbb{Z}^d$ to a finite set as follows.

\begin{defn}\label{def:autoZ} {\em Let $\Delta$ be a finite set. 
A function $f: \mathbb{Z}^d \rightarrow \Delta$ 
is $k$-\emph{automatic} if 
there is a finite automaton 
that takes words
over ${\mathcal L}_k(d)$ as input and has the property that reading a word $W\in \mathcal{L}_k(d)$,  
the automaton outputs $f([W]_k)$. 

Similarly,  a  subset $S$ of $\mathbb{Z}^d$ is $k$-automatic if 
its characteristic function, $f:\mathbb{Z}^d\to \{0,1\}$, 
defined by $f(n_1,\ldots ,n_d)=1$ if $(n_1,\ldots ,n_d)\in S$; and $f(n_1,\ldots ,n_d)=0$, 
otherwise, is $k$-automatic.}
\end{defn}

In fact, much as in the classical situation, automaticity of subsets of $\mathbb{Z}^d$ 
can be characterized using the kernel.

\begin{defn} {\em 
Let $d\geq 1$ be an integer and  $\Delta$  a finite set.  
Given a map $f:\mathbb{Z}^d\to \Delta$, we define the $k$-kernel of $f$ 
to be the collection of all maps of the form
\begin{displaymath}
g(n_1,\ldots ,n_d):=f(k^a n_1+b_1,\ldots ,k^a n_d+b_d)
\end{displaymath}
 where $a\ge 0$ and $0\le b_1,\ldots ,b_d<k^a$. }
\end{defn}

\begin{prop}\label{prop:autoZequiv} Let $d\geq 1$ be an integer and  $\Delta$  a finite set. Given a map 
$f:\mathbb{Z}^d\to \Delta$, the following are equivalent.
\begin{enumerate}
\item[{\rm (i)}] The map $f$ is $k$-automatic.
\item[{\rm (ii)}]  The $k$-kernel of $f$ is finite. 
\item[{\rm (iii)}]  For each ${\bf \epsilon}=(\epsilon_1,\ldots ,\epsilon_d) \in \{\pm 1\}^d$, the function
$f_{{\bf \epsilon}}:\mathbb{N}^d\to \Delta$ defined by 
$(n_1,\ldots ,n_d)\mapsto f(\epsilon_1 n_1,\ldots ,\epsilon_d n_d)$ is $k$-automatic in the usual sense.
\end{enumerate}
\end{prop}

\begin{proof}
We note that by definition of automatic maps on $\mathbb{Z}^d$, each of the 
$f_{{\bf \epsilon}}$ is $k$-automatic in the usual sense and hence (i) implies (iii).  
Similarly, (iii) implies (i).  Next assume that (iii) holds.  
Let $h(n_1,\ldots ,n_d)=f(k^an_1+b_1,\ldots , k^a n_d +b_d)$ be a map in the kernel of $f$.  
Then for ${\bf \epsilon}=(\epsilon_1,\ldots ,\epsilon_d) \in \{\pm 1\}^d$, the map
$h_{{\bf \epsilon}}:\mathbb{N}^d\to \Delta$ defined by $(n_1,\ldots ,n_d)\mapsto 
h(\epsilon_1 n_1,\ldots \epsilon_d n_d)$ is of the form
$$f(\epsilon_1k^a n_1+b_1,\ldots ,\epsilon_d k^a n_d + b_d),$$ which is in the 
$k$-kernel of $f_{{\bf \epsilon}}$.  Since there are only finitely many
 ${\bf \epsilon}=(\epsilon_1,\ldots ,\epsilon_d) \in \{\pm 1\}^d$ and only finitely many elements
  in the kernel of $f_{{\bf \epsilon}}$, we see that the kernel of $f$ is finite and hence (iii) implies (ii).  
  Similarly, (ii) implies (iii).  
\end{proof}

\subsection{Automatic subsets of finitely generated abelian groups}

We introduce here a relevant notion of automaticity for subsets of arbitrary finitely generated 
abelian groups. In this area, we quote \cite{Aandal} where the authors provide a 
general framework for the automaticity of maps from some semirings to finite sets.   
In particular, a similar notion of automaticity for subsets 
of $\mathbb Z^2$ was considered in that paper.

\medskip

In this more general framework, it seems more natural to define first $k$-automatic maps in terms of some generalized $k$-kernels and then to prove that such maps can be  characterized in terms of finite automata.

\medskip

In the rest of this section, all finitely generated abelian groups are written additively. 
We thus first define the $k$-kernel of a map from a finitely generated abelian group to a finite set.

\begin{defn}
{\em Let $\Gamma$ be a finitely generated abelian group and   
$T=\{\gamma_1,\ldots ,\gamma_d\}$  a set of generators of $\Gamma$.  
Let  $\Delta$ be a finite set. Given a map 
$f:\Gamma\to \Delta$, we define the $k$-kernel  of $f$ with respect to the generating set 
$T$ to be the collection of all maps from $\Gamma$ to $\Delta$ of the form
\begin{displaymath}
g(x):=f(k^a x+ b_1 \gamma_1+\cdots +b_d \gamma_d)
\end{displaymath}
such that $a\ge 0$ and $0\le b_1,\ldots ,b_d<k^a$.  }  \end{defn}

We can now define $k$-automatic maps as follows.

\begin{defn}
{\em  Let $\Gamma$ be a finitely generated abelian group and $\Delta$ a finite set. 
A map 
$f:\Gamma\to \Delta$ is $k$-automatic if its $k$-kernel with respect to every finite generating 
set of $\Gamma$ is finite.}  \end{defn}

As usual, we can use the previous definition to introduce the notion of a $k$-automatic subset 
of a finitely generated abelian group. 

\begin{defn}\label{def:autofgg}
{\em  Let $\Gamma$ be a finitely generated abelian group.  
A subset $S$  of $ \Gamma$ is $k$-automatic if the map $\chi_S:\Gamma\to \{0,1\}$, 
defined by $\chi_S(x)=1$ if and only if $x\in S$, is $k$-automatic.   }
\end{defn}

We note that our definition of $k$-automaticity appears to be somewhat difficult to verify, 
as we must check that the $k$-kernel is finite with respect to every finite generating set.  
As shown below, it actually suffices to check that the $k$-kernel is finite 
with respect to just anyone generating set.

\begin{prop}
Let $\Gamma$ be a finitely generated abelian group and $\Delta$ a finite set.  
Let us assume that the map $f:\Gamma\to \Delta$ has a finite $k$-kernel with respect to some generating set 
of $\Gamma$.  
Then the map $f$ is $k$-automatic.
\label{prop: independentofcoice}
\end{prop}

\begin{proof} Suppose that the $k$-kernel of $f$ is finite with respect to the generating set
$T:=\{\gamma_1,\ldots ,\gamma_d\}$ of $\Gamma$ and 
let $f_1,\ldots ,f_m$ denote the distinct maps in the $k$-kernel of $f$.  

Given another generating set of $\Gamma$, say $T':=\{\delta_1,\ldots ,\delta_e\}$, we have to show 
that the $k$-kernel of $f$ with respect to $T'$ is also finite.   

There exist integers $c_{i,j}$ with $1\le i\le d$ and $1\le j\le e$ such that 
$$\delta_j =\sum_{i=1}^d c_{i,j} \gamma_i$$ for $j\in \{1,\ldots ,e\}$.  
Set $N:=\sum_{i,j} |c_{i,j}|$.  
Given an integer $i$, $1\leq i\leq m$, and a  $d$-tuple of integers ${\bf j}=(j_1,\ldots ,j_d)$, 
we define the map $g_{i,{\bf j}}$ from $\Gamma$ to $\Delta$ by 
$$g_{i,{\bf j}}(x):=f_i( x + j_1\gamma_1+\cdots + j_d \gamma_d)$$ for all $x\in \Gamma$.  
We claim that the $k$-kernel of $f$ with respect to $T'$ is contained in the finite set $\mathcal S$ 
defined by 
$${\mathcal S}:= \left\{ g_{i,{\bf j}} : \Gamma \to\Delta \mid {\bf j}=(j_1,\ldots ,j_d)\in 
\{-N,-N+1,\ldots ,N\}^d, \; i\in \{1,\ldots ,m\}\right\}\,.$$  

To see this, note that if $a\ge 0$ and $0\le b_1,\ldots ,b_d<k^a$, then
$$b_1 \delta_1+\cdots + b_e \delta_e=b_1' \gamma_1+\cdots +b_d'\gamma_d\,,$$ 
where $\displaystyle b_i' = \sum_{j=1}^{e} b_j c_{i,j}$.  It follows that 
$$\vert b'_i\vert \leq N(k^a-1)$$ for every $i$, $1\leq i\leq d$.  We can thus write 
$b_i' = k^a m_i + r_i$ with $\vert m_i\vert <N$ and $0\le r_i < k^a$.  This implies that 
\begin{eqnarray*}
f(k^a x + b_1 \delta_1+\cdots + b_e \delta_e) &=&
f(k^a x + b_1' \gamma_1+\cdots +b_d'\gamma_d) \\
&= &f(k^a(x+m_1\gamma_1+\cdots + m_d\gamma_d)\\
&&\;\;\;+r_1\gamma_1+\cdots+r_d\gamma_d)\\
&= &f_{\ell}(x+m_1\gamma_1+\cdots + m_d\gamma_d)
\end{eqnarray*}
for some $\ell$, $1\leq \ell\leq m$.  
Thus we see that $$f(k^a x + b_1 \delta_1+\cdots + b_e \delta_e)=
g_{\ell, {\bf m}}(x)$$ where ${\bf m}:=(m_1,\ldots,m_d)$, which proves that the 
$k$-kernel of $f$ with respect to the generating set $T'$ is included in the finite 
set ${\mathcal S}$, as claimed.
\end{proof}

\begin{prop}\label{prop:automorph}
Let $\Gamma_1$ and $\Gamma_2$ be two finitely generated abelian groups, 
and $\Phi:\Gamma_1\to\Gamma_2$ a surjective group homomorphism. If $S$ is a $k$-automatic subset of $\Gamma_2$ then $\Phi^{-1}(S)$ is a $k$-automatic subset of $\Gamma_1$.
\end{prop}

\begin{proof}
 Let $f$ and $g$ denote respectively the characteristic function of $\Phi^{-1}(S)$ and $S$. Let 
$\{\gamma_1,\ldots ,\gamma_d\}$ be a set of generators of $\Gamma_1$.  
Then if $a\geq0$ and $0\le b_1,\ldots ,b_d< k^a$, we infer from the definition of $f$ that 
$$f(k^a x+b_1\gamma_1+\cdots +b_d\gamma_d)=1\iff
\Phi(k^a x+b_1\gamma_1+\cdots + b_d\gamma_d)\in S\, ,$$ which occurs if and only if
$$g\left(k^a \Phi(x)+\sum_{i=1}^d b_i \Phi(\gamma_i)\right) =1\, .$$   

Note that $$T:= \{\Phi(\gamma_i)~:~1\le i\le d\}$$ is a set of generators of $\Gamma_2$ since $\Phi$ 
is surjective. Since, by assumption, $g$ is $k$-automatic, the $k$-kernel of $g$ is finite with respect to 
$T$. Thus the $k$-kernel of $f$ is finite with respect to $T':=\{\gamma_1,\ldots ,\gamma_d\}$. 
The result now follows from Proposition \ref{prop: independentofcoice}.
\end{proof}

We can now prove, as we may expect, 
that a $k$-automatic subset of a finitely generated abelian group can be described by 
a finite automaton. 

 \begin{prop}\label{prop:fggcarac}
 Let $\Gamma$ be a finitely generated group, $\{\gamma_1,\ldots,\gamma_d\}$ a set of generators 
 of $\Gamma$, and $S$ a  subset of $\Gamma$. Then $S$ is $k$-automatic if and only if there exists a 
 finite automaton that  takes words over ${\mathcal L}_k(d)$ as input and has the property that for 
 every $d$-tuple of integers $(n_1,\ldots,n_d)$ the word  
 $[(n_1,\ldots,n_d)]_k^{-1}\in {\mathcal L}_k(d)$ is accepted by the automaton if and only if 
 $n_1\gamma_1 + \cdots + n_d\gamma_d$ belongs to $S$. 
 \end{prop}
 
 \begin{proof}
 For every integer $i$, $1\leq i\leq d$, we denote by $e_i:= (0, 0, . . . , 0, 1, 0 . . . , 0)$ the element of 
 $\mathbb Z^d$ whose $j$th coordinate is $1$ and whose other coordinates are $0$.  
 Let $\Phi$ be the surjective group homomorphism from $\mathbb Z^d$ to $\Gamma$ defined by 
 $\Phi(e_i)=\gamma_i$ for every integer $i$, $1\leq i\leq d$.  
 
 If $S$ is $k$-automatic then, by Proposition \ref{prop:automorph}, $\Phi^{-1}(S)$ is a $k$-automatic subset of 
 $\mathbb Z^d$. By Definition \ref{def:autoZ}, there is a finite automaton that takes words
over ${\mathcal L}_k(d)$ as input and has the property that the word $W\in \mathcal{L}_k(d)$ is 
accepted by the automaton if and only if $[W]_k$ belongs to $\Phi^{-1}(S)$. Thus for every $d$-tuple 
 of integers $(n_1,\ldots,n_d)$ the word  
 $[(n_1,\ldots,n_d)]_k^{-1}\in {\mathcal L}_k(d)$ is accepted by this automaton if and only if 
 $n_1\gamma_1+\cdots + n_d\gamma_d$ belongs to $S$. 
 
On the other hand, if there exists  a finite automaton such that for every $d$-tuple 
 of integers $(n_1,\ldots,n_d)$ the word  
 $[(n_1,\ldots,n_d)]_k^{-1}\in {\mathcal L}_k(d)$ is accepted by this automaton if and only if 
 $n_1\gamma_1+\cdots +n_d\gamma_d$ belongs to $S$. The same automaton can also be 
 used to  recognize $\Phi^{-1}(S)$. Thus  $\Phi^{-1}(S)$ is a $k$-automatic subset of $\mathbb Z^d$. 
 By Proposition \ref{prop:autoZequiv}, the set $\Phi^{-1}(S)$ has a finite $k$-kernel and it follows 
 that $S$ has a finite $k$-kernel with respect to $\{\gamma_1,\ldots,\gamma_d\}$. By Proposition  
 \ref{prop: independentofcoice}, $S$ 
 is thus a $k$-automatic subset of $\Gamma$.
 \end{proof}
\section{Proof of our main result}
\label{proof}

Our aim is to prove Theorem \ref{thm: main}.  
Throughout this section, we take $d$ to be a natural number.  
We let ${\bf n}$ and ${\bf j}$ denote respectively the $d$-tuple of natural numbers 
$(n_1,\ldots ,n_d)$ and $(j_1,\ldots ,j_d)$.  We will also let ${\bf t }^{\bf n}$ denote 
the monomial $t_1^{n_1}\cdots t_d^{n_d}$ in indeterminates $t_1,\ldots ,t_d$. The degree 
of such a monomial is the nonnegative integer $n_1+\cdots + n_d$. Given a 
polynomial $P$ in $K[{\bf t}]$, we denote by $\deg P$ the maximum of the degrees 
of the monomials appearing in $P$ with nonzero coefficient.

\begin{defn} {\em
We say that a power series $f({\bf t}) \in K[[{\bf t}]]$ 
is algebraic if it is algebraic over the field of rational functions $K({\bf t})$, 
that is,  if there exist polynomials $A_0,\ldots, A_m\in {K}[{\bf t}]$, not all zero, 
such that
\begin{displaymath}
\sum_{i=0}^{m} A_i({\bf t})f({\bf t})^i \ = \ 0\, .
\end{displaymath}}
\end{defn}

In order to prove Theorem \ref{thm: main} we need to introduce some notation.   
For each 
 ${\bf j}=(j_1,\ldots ,j_d) \in \{0,1,\ldots ,p-1\}^d$, we define 
 $e_{\bf j}:\mathbb{N}^d\to \mathbb{N}^d$ by 
 \begin{equation}\label{AB:equation:ej}
 e_{\bf j}(n_1,\ldots ,n_d):=(pn_1+j_1,\ldots ,pn_d+j_d)\, .
 \end{equation}
 We let $\Sigma$ denote the semigroup generated by the collection of all 
 $e_{\bf j}$ under composition. In view of Definition \ref{defn: kerd}, this semigroup 
 is intimately related to the definition of the $p$-kernel of $d$-dimensional maps. 
 As a direct consequence of Theorem \ref{thm: Salon}, we make the following remark 
which underlines the important role that will be played by the semigroup $\Sigma$ 
in the proof of Theorem \ref{thm: main}.

\begin{rem}{\em  Let $\Delta$ be a finite set. Then a map $a:\mathbb{N}^d\to \Delta$ 
is $p$-automatic if and only the set of functions $\{a\circ e~\mid~e\in\Sigma\}$ is a finite set.}
\label{rem: kernel}
\end{rem}

We recall that a field $K$ of characteristic $p>0$ is perfect if the 
map $x\mapsto x^p$ is surjective on $K$.  
Let $p$ be a prime number and let $K$ be a perfect field of characteristic $p$.  For 
every ${\bf j}\in \Sigma_p^d=\{0,1,\ldots ,p-1\}^d$, we define the so-called Cartier operator 
$E_{\bf j}$ from $K[[{\bf t}]]$ into itself by 
\begin{equation}\label{AB:equation:EJ}
E_{\bf j}(f({\bf t}))\ := \ \sum_{{\bf n}\in \mathbb{N}^d} (a\circ e_{\bf j}({\bf n}))^{1/p}{\bf t}^{\bf n}
\end{equation}
where  $f({\bf t}):= \sum_{{\bf n} \in \mathbb N^d} a({\bf n}) {\bf t}^{\bf n}
\in K[[{\bf t}]]$.  Then we have the following useful decomposition: 
\begin{equation}\label{eq: fs}
f  = \sum_{{\bf j}\in \Sigma_p^d} {\bf t}^{\bf j} E_{\bf j}( f)^p \, .
\end{equation}

 We now recall the following simple classical result, usually known as Ore's lemma.
 
 \begin{lem}\label{lem: ore}
 Let $f({\bf t})\in K[[{\bf t}]]$ be a nonzero algebraic power series. Then there exists 
a positive integer $r$ and polynomials $P_0,\ldots,P_r$ in $\mathbb K[{\bf t}]$ such that 
$$
\sum_{i=0}^r P_if^{p^i} = 0 
$$
and $P_0\not=0$.
 \end{lem}
 
 \begin{proof}
 Since $f$ is algebraic, $\left\{ f,f^p,f^{p^2},\ldots\right\}$ is linearly dependent over $K({\bf t})$. 
 There thus exists a natural number $r$ and polynomials $P_0,\ldots,P_r$ in $\mathbb K[{\bf t}]$ 
 such that 
$$
\sum_{i=0}^r P_if^{p^i} = 0 \, .  
$$
It remains to prove that one can choose $P_0\not = 0$.  Let $k$ be the smallest nonnegative 
integer such that $f$ satisfies a relation of this type with $P_k\not=0$. We shall prove that $k=0$ 
which will end the proof. 
We assume that $k>0$ and we argue by contradiction.   Since $P_k\not=0$, we infer from 
Equality (\ref{eq: fs}) that there exists 
a $d$-tuple ${\bf j}\in \Sigma_p^d$ such that $E_{\bf j}(P_k)\not =0$.  Since 
$\sum_{i=k}^{r} P_if^{p^i}=0$, we have 
$$
E_{\bf j}\left(\sum_{i=k}^{r} P_if^{p^i}\right)= \sum_{i=k}^{r} E_{\bf j}\left(P_if^{p^i}\right)
= \sum_{i=k}^{r} E_{\bf j}\left(P_i\right) f^{p^{i-1}}= 0 \, .
$$
We thus obtain a new relation of the same type but for which the coefficient of 
$f^{p^{k-1}}$ is nonzero. This provides a contradiction with the definition of $k$.  
\end{proof}
 
\medskip

We now let $\Omega$ denote the semigroup generated 
 by the collection of the Cartier operators $E_{\bf j}$ and the identity operator 
 under composition.  We let $\Omega(f)$ denote the orbit of $f$ under the action of 
 $\Omega$, that is, 
 $$
 \Omega(f) := \left\{ E(f) \mid E \in \Omega \right\} \, .
 $$
 As in the work of Harase \cite{Har88} and of Sharif and Woodcock \cite{SW}, the $K$-vector space 
 spanned by $\Omega(f)$ will play an important role. 
We will in particular need the following auxiliary result based on Ore's lemma.  

\begin{lem} 
Let $K$ be a perfect field of characteristic $p$,  
and let 
$$f({\bf t}):=\sum_{{\bf n}\in \mathbb{N}^d} a({\bf n}){\bf x}^{\bf n} 
\in K[[{\bf t}]]$$ be  a nonzero algebraic function 
over $K({\bf t})$.  Then there exists a natural number $m$ and 
there exist maps $a_1,\ldots ,a_m:\mathbb{N}^d\to K$ with the following properties. 

\begin{itemize}

\item[\textup{(i)}] The formal power series $f_i({\bf t}):=\sum_{{\bf n}\in \mathbb{N}^d} 
a_i({\bf n}){\bf t}^{\bf n}$, $1\le i\le m$,  form a 
basis of the K-vector space spanned by $\Omega(f)$.

\item[\textup{(ii)}] One has $f_1=f\, .$

\item[\textup{(iii)}]  Let  
$g({\bf t}) := \sum_{ {\bf n} \in \mathbb N^d } b( {\bf n} ) {\bf t}^{\bf n}$ be a power series that 
belongs to $\Omega(f)$. Then $b\circ e_{\bf j} \in K\,a_1^p+\cdots + K\,a_m^p\,$ for every 
${\bf j}\in\{0,\ldots, p-1\}^d \,.$  
\end{itemize}\label{lem: SW}
\end{lem}
 
\begin{proof} 
Let $f({\bf t})\in K[[{\bf t}]]$ be a nonzero algebraic power series. By Lemma \ref{lem: ore}, 
there exist a positive integer $r$ and polynomials $P_0,\ldots,P_r$ in $\mathbb K[{\bf t}]$ 
such that 
$$
\sum_{i=0}^r P_if^{p^i} = 0 
$$
and $P_0\not=0$. Set $\tilde{f} := P_0^{-1} f$. Then 
\begin{equation}\label{eq: g}
\tilde{f} = \sum_{i=1}^r Q_i \tilde{f} ^{p^i} \, ,
\end{equation}
where $Q_i= -P_i P_0^{p^i-2}$. 
Set $M := \max \{ \deg P_0, \deg Q_i \mid 1 \leq i \leq r \}$ and 
\begin{equation}\label{eq: H}
{\mathcal H} := \left\{ h \in K(({\bf t})) \mid h = \sum_{i=0}^r R_i\tilde{f} ^{p^i} \mbox{ such that } 
R_i \in K[{\bf t}] \mbox{ and } \deg R_i \leq M \right\} \, .
\end{equation}
We first note that $f$ belongs to $\mathcal H$ since $f= P_0\tilde{f} $ and $\deg P_0\leq M$. 
We also observe that ${\mathcal H}$ is closed under the action of $\Omega$. Indeed, 
if $h:=\sum_{i=0}^r R_i\tilde{f} ^{p^i}  \in {\mathcal H}$ and ${\bf j}\in \{0,\ldots,p-1\}^d$, then 
$$\begin{array}{ll}
E_{\bf j}(h) & \displaystyle = E_{\bf j}\left(R_0\tilde{f}  + \sum_{i=1}^r R_i\tilde{f} ^{p^i}\right )
=E_{\bf j}\left(\sum_{i=1}^r (R_0Q_i+R_i)\tilde{f} ^{p^i}\right )\\ 
&\displaystyle  = \sum_{i=1}^r E_{\bf j}(R_0\tilde{f}  + R_i)\tilde{f} ^{p^{i-1}} \, ,
\end{array}$$
and since $\deg (R_0Q_i+R_i) \leq 2M$, we have $\deg E_{\bf j}(R_0Q_i+R_i) \leq 2M/p\leq M$. 
It follows that the $K$-vector space spanned by $\Omega(f)$ is contained in ${\mathcal H}$ 
and thus has finite dimension, say $m$.

 We can thus pick 
maps $a_1,\ldots ,a_m:\mathbb{N}^d\to K$ such that the $m$ power series
$f_i({\bf t}):=\sum_{{\bf n}\in \mathbb{N}^d} a_i({\bf n}){\bf t}^{\bf n}$  form a basis of 
$\Omega(f)$. Furthermore, since by assumption 
$f$ is a nonzero power series, we can chose $f_1=f$.  
Let $b:\mathbb{N}^d\to K$ be such that 
$g({\bf t}):=\sum_{{\bf n}\in \mathbb{N}^d} b({\bf n}){\bf t}^{\bf n}$ belongs to $\Omega(f)$. 
Observe that the power series $g$ can be decomposed as 
\begin{equation} \label{AB:equation:eq1}
g({\bf t}) = \sum_{{\bf j}\in \{0,\ldots ,p-1\}^d} {\bf t}^{\bf j} E_{\bf j}(g({\bf t}))^p \, .
\end{equation}
  By assumption,
$E_{\bf j}(g({\bf t}))\in K\, f_1({\bf t})+\cdots + K\, f_m({\bf t})$ 
and hence
$E_{\bf j}(g({\bf t}))^p\in K\, f_1({\bf t})^p+\cdots + K\, f_m({\bf t})^p$.  
Let ${\bf j}\in\{0,1,\ldots, p-1\}^d$.  Considering the coefficient of ${\bf t}^{p{\bf n}+{\bf j}}$ in 
Equation (\ref{AB:equation:eq1}), we see that 
$b\circ e_{\bf j}({\bf n})$ is equal to the coefficient of ${\bf t}^{p{\bf n}}$ in $E_{\bf j}(g({\bf t}))^p$, 
which belongs to $K\, a_1({\bf n})^p+\cdots + K\, a_m({\bf n})^p$. This concludes the proof. 
\end{proof}

\medskip

We will also need the following lemma that says we will only have to work with finitely 
generated extensions of the prime field instead of general fields of characteristic $p$.

\begin{lem} \label{lem: fg}  Let $f_1,\ldots, f_m$ be 
power series as in Lemma \ref{lem: SW}. Then there 
is a finitely generated field extension $K_0$ of $\mathbb{F}_p$  such that 
all coefficients of the power series $f_1,\ldots, f_m$ belong to $K_0$.
\end{lem}

\begin{proof} Let $\tilde{f} :=\sum_{{\bf n}\in \mathbb{N}^d} \tilde{a} ({\bf n}){\bf t}^{\bf n}$ 
be defined as in Equation (\ref{eq: g}), that is,  
\begin{equation}\label{eq: g'}
\tilde{f} = \sum_{i=1}^r Q_i \tilde{f} ^{p^i} \, ,
\end{equation}
Let also $\mathcal H$ be the $K$-vector space 
defined as in Equation (\ref{eq: H}), that is, 
\begin{equation}\label{eq: H'}
{\mathcal H} = \left\{ h \in K(({\bf t})) \mid h = \sum_{i=0}^r R_i\tilde{f} ^{p^i} \mbox{ such that } 
R_i \in K[{\bf t}] \mbox{ and } \deg R_i \leq M \right\} \, .
\end{equation}
Since $\mathcal H$ contains the $K$-vector space spanned by $\Omega(f)$, 
the power series $f_1,\ldots, f_m$ belong to ${\mathcal H}$.  There thus exist a finite 
number of polynomials $R_{i,k}$ such that 
$$
f_k = \sum_{i=0}^r R_{i,k}\tilde{f} ^{p^i} \, .
$$
It thus remains to prove that there exists a finitely generated field extension $K_0$ of 
$\mathbb{F}_p$  such that all coefficients of $\tilde{f} $ belong to $K_0$. Indeed, by adding to $K_0$ 
all the coefficients of the polynomials $R_{i,k}$, we would obtain a finitely generated field extension 
$K_1$ of $\mathbb{F}_p$ such that all coefficients of the power series $f_1,\ldots, f_m$ belong to 
$K_1$.

Given a $d$-tuple ${\bf n}=(n_1,\ldots ,n_d)$, we set $\Vert {\bf n} \Vert:= \max (n_1,\ldots ,n_d)$. 
Let $N$ be a positive integer. We let $K_0$ be the finitely generated extension of $\mathbb{F}_p$ 
generated by the coefficients of $Q_1,\ldots ,Q_r$ and the collection of coefficients of ${\bf t}^{\bf n}$ in 
$\tilde{f} ({\bf t})$ with $\Vert {\bf n} \Vert \le N$.  We claim that the coefficients of $\tilde{f} $ all lie in $K_0$.  
We prove by induction on $\Vert {\bf n} \Vert$ that all coefficients $\tilde{a} ({\bf n})$ belongs to $K_0$. 
By construction, this holds whenever $\Vert {\bf n} \Vert \le N$.  

Suppose that the claim holds whenever $\Vert {\bf n} \Vert< M$ for some $M>N$ and 
let us assume that  $\Vert {\bf n} \Vert= M$. 
Then if we consider the coefficient of $t_1^{ n_1}\cdots t_d^{n_d}$ in both sides of Equation 
\ref{eq: g'}, we get that 
$$
\tilde{a} (n_1,\ldots ,n_d) \in \sum_{i = 1}^r \sum_{(m_1,\ldots ,m_d)\in S} K_0 \tilde{a} (m_1,\ldots ,m_d)^{p^i},
$$ 
where $S$ is the (possibly empty) set of all $d$-tuples ${\bf m}:=(m_1,\ldots ,m_d)\in \mathbb{N}^d$ 
such that either $m_i=0$ or $m_i<n_i$ for each $i\in \{1,\ldots, d\}$. Since $M>0$, we get that 
$\Vert {\bf m}\Vert < M$ and 
the inductive hypothesis implies that  
$$
\sum_{i =1}^r \sum_{(m_1,\ldots ,m_d)\in S} K_0 \tilde{a} (m_1,\ldots ,m_d)^{p^i}\subseteq K_0 \, ,
$$ 
and so $\tilde{a} (n_1,\ldots ,n_d)\in K_0$.  This completes the induction and shows that all coefficients 
of $\tilde{f} $ lie in $K_0$. 
 \end{proof}

Before proving Theorem \ref{thm: main}, we first fix a few notions.  
Given a finitely generated field extension $K_0$ of $\mathbb{F}_p$, we 
let $K_0^{\langle p \rangle}$ denote the subfield consisting of all elements 
of the form $x^p$ with $x\in K_0$.  Given $\mathbb{F}_p$-vector subspaces 
$U$ and $V$ of $K_0$ we let $VU$ denote the $\mathbb{F}_p$-subspace of 
$K_0$ spanned by all products of the form $vu$ with $v\in V,u\in U$.  
We let $V^{\langle p \rangle}$ denote the $\mathbb{F}_p$-vector subspace consisting 
of all elements of the form $v^p$ with $v\in V$.  We note that since $K_0$ is a 
finitely generated field extension of $\mathbb{F}_p$, $K_0$ is a finite-dimensional 
$K_0^{\langle p \rangle}$-vector space.  If we fix a basis
\begin{displaymath}
K_0=\bigoplus_{i=1}^r  K_0^{\langle p \rangle}h_i
\end{displaymath}
then we have \emph{projections} $\pi_1,\ldots ,\pi_r :K_0\to K_0$
defined by 
\begin{equation} \label{eq: 2}x=\sum_{i=1}^r \pi_i(x)^p h_i\, .\end{equation}

\begin{rem} \label{{AB:remark:rem2}}{\em
For $1\le i\le r$ and $x,y,z\in K_0$ we have
\begin{displaymath}
\pi_i(x^p y+z) = x\pi_i(y)+\pi_i(z) \, .
\end{displaymath}}
\end{rem}

The last ingredient we have to state before proving Theorem \ref{thm: main} 
is a rather technical result, but very useful,  due to Derksen, which we state here without proof. 
It corresponds to Proposition 5.2 in \cite{Der}. 
Basically,  
we will prove an effective version of this result later in Section \ref{section: eff} 
(step $2$ in the proof of Theorem \ref{thm: eff}).

\begin{prop}[Derksen]\label{AB:proposition:derksen} 
Let $K_0$ be a finitely generated field extension of $\mathbb{F}_p$ and let
$\pi_1,\ldots ,\pi_r :{K}_0\to {K}_0$ be as in Equation (\ref{eq: 2}).  
If $U$ is a finite-dimensional $\mathbb{F}_p$-vector subspace of $K_0$.  
Then there exists a finite-dimensional $\mathbb{F}_p$-vector subspace 
$V$ of $K_0$ containing $U$ such that $$\pi_i(VU)\subseteq V$$ 
for all $i$ such that $1\le i\le r$.
\end{prop}

We are now ready to prove Theorem \ref{thm: main}.

\begin{proof}[Proof of Theorem \ref{thm: main}]
By enlarging $K$ if necessary, we may assume that $K$ is perfect.
By Lemma \ref{lem: SW} we can find maps 
$a_1,\ldots ,a_m:\mathbb{N}^d\to K$ with the following properties. 
\begin{enumerate}

\item[\textup{(i)}]  The power series $f_i({\bf t}):=\sum_{{\bf n}\in \mathbb{N}^d} 
a_i({\bf n}){\bf t}^{\bf n}$, $1\le i\le m$, 
form a basis of the $K$-vector space spanned by $\Omega(f)$.

\item[\textup{(ii)}]  One has $f_1=f$.

\item[\textup{(iii)}]  Let  
$g({\bf t}) := \sum_{ {\bf n} \in \mathbb N^d } b( {\bf n} ) {\bf t}^{\bf n}$ be a power series that 
belongs to $\Omega(f)$. Then $b\circ e_{\bf j} \in K\,a_1^p+\cdots + K\,a_m^p\,$ for every 
${\bf j}\in\{0,\ldots, p-1\}^d \,.$  

\end{enumerate}

In particular, given $1\leq i \leq m$ and ${\bf j} \in \{0,1,\ldots ,p-1\}^d$, 
there are elements $ \lambda(i,{\bf j},k)$,  $1\le k\le m$,  such that
\begin{equation} \label{{AB:equation:aiej}}
a_i\circ e_{{\bf j}} = \sum_{k=1}^m \lambda(i,{\bf j},k) a_k^p\, .
\end{equation}
Furthermore, by Lemma \ref{lem: fg},  there exists a finitely generated 
field extension of $\mathbb{F}_p$ such that all coefficients of $f_1,\ldots ,f_m$ 
are contained in this field extension.  It follows that the subfield $K_0$ of 
$K$ generated by the coefficients of $f_1({\bf t}),\ldots , f_m({\bf t})$ 
and all the elements $\lambda(i,{\bf j},k)$  
is a finitely generated field extension of $\mathbb{F}_p$.  

Since $K_0$ is a finite-dimensional $K_0^{\langle p \rangle}$-vector space, 
we can fix a basis $\{h_1,\ldots ,h_r\}$ of $K_0$, that is, 
\begin{displaymath}
K_0=\bigoplus_{i=1}^r  K_0^{\langle p \rangle}h_i \,.
\end{displaymath}
As already mentioned,  we have projections $\pi_1,\ldots ,\pi_r : K_0\to K_0$
defined by 
\begin{equation} \label{{AB:equation:projdecomp}}
x=\sum_{i=1}^r \pi_i(x)^p h_i \,.
\end{equation}
We let $U$ denote the finite-dimensional $\mathbb{F}_p$-vector subspace of 
$K_0$ spanned by the elements 
$\lambda(i,{\bf j},k)$, $1\le i,k\le m$ and ${\bf j} \in \{0,1,\ldots ,p-1\}^d$, and by $1$.  
By Equation (\ref{{AB:equation:aiej}}), we have
\begin{equation} \label{{AB:equation:aiej2}}
a_i\circ e_{{\bf j}} \in  U a_1^p+\cdots +U a_m^p \,,
\end{equation}
for $1\le i\le m$ and ${\bf j} \in \{0,1,\ldots ,p-1\}^d$.
By Proposition \ref{AB:proposition:derksen} there exists a finite-dimensional 
$\mathbb{F}_p$-vector subspace $V$ of $K_0$ containing $U$ such that 
$\pi_i(VU)\subseteq V$ for $1\le i\le r$. 

We now set 
\begin{displaymath}
W:=Va_1+\cdots +Va_m\subseteq \{b~\mid b:\mathbb{N}^d\to K_0\} \,.
\end{displaymath}
We note that since $V$ is a  finite-dimensional $\mathbb{F}_p$-vector space, it is a finite set. 
It follows that $W$ is also a finite set since  
$\mbox{ Card } W \le (\mbox{ Card } V)^d<\infty$.  Note also that if 
  $\ell \in \{1,\ldots ,r\}$, $i\in \{1,\ldots ,m\}$, and $j\in \{0,1,\ldots ,p-1\}^d$ then by Equation (\ref{{AB:equation:aiej2}}) and Remark \ref{{AB:remark:rem2}} we have
\begin{align}  
\pi_{\ell}(Va_i\circ e_{\bf j})& \subseteq \pi_{\ell}(VUa_1^p+\cdots + VUa_m^p) \nonumber \\ 
& \subseteq   \pi_{\ell}(VU)a_1 + \cdots +  \pi_{\ell}(VU)a_m \nonumber \\ 
&\subseteq  Va_1 + \cdots + Va_m   \, . \nonumber
\end{align} 
By Remark \ref{{AB:remark:rem2}}, we obtain that 
\begin{equation}\label{eq: bl}
b_{\ell}:=\pi_{\ell}(b\circ e_{\bf j})\in W
\end{equation}
for all $b\in W$, 
${\bf j} \in \{0,1,\ldots ,p-1\}^d$, and $1\le \ell \le r$. 
Since  $\{h_1,\ldots ,h_r\}$ form a basis of $K_0$ as a $K_0^{\langle p \rangle}$-vector space, 
given $x$ in $K_0$, we have 
$$
x=0 \iff (\pi_{\ell}(x)=0 \mbox{ for all } 1\le \ell \le r) \, .
$$
In particular, 
\begin{equation}\label{eq: base}
b(p{\bf n}+{\bf j})=0 \iff b_1({\bf n})=b_2({\bf n})=\cdots =b_r({\bf n})=0 \, .
\end{equation}

Given a map $b:\mathbb{N}^d\to K_0$, we define the map  
$\chi_b:\mathbb{N}^d\to \{0,1\}$ by 
\begin{equation}\label{eq: chi}
 \chi_b({\bf n}) \ = \ 
 \left\{ 
 \begin{aligned} 
 0 & \;{ \rm if } \; b({\bf n}) \not = 0 \,  \\
1 &\;{ \rm if } \;~b({\bf n})=0 \, . 
\end{aligned} 
\right. 
\end{equation}
Then we set 
\begin{equation*}
X := \{ \chi_{b_1}\cdots \chi_{b_t}~\mid~t\ge 0, b_1,\ldots , b_t\in W\} \, .
\end{equation*} 
We first get from Equation (\ref{eq: base}) that 
\begin{displaymath}
(\chi_{b}\circ e_{\bf j})({\bf n})=\prod_{\ell=1}^r \chi_{b_{\ell}}({\bf n})\, .
\end{displaymath}
Furthermore, we infer from Equation \ref{eq: bl}  that  
 $b_{\ell}\in W$ for all $b\in W$, ${\bf j}\in \{0,1,\ldots ,p-1\}^d$, and $1\le \ell \le r$.  
The definition of $X$ then implies that $\chi_{b}\circ e_{\bf j}$ belongs to $ X$. 
More generally, it follows  that  
\begin{equation}\label{eq: S}
\forall \chi \in X, \forall e\in \Sigma, \;\;\chi \circ e\in X
\,  .
\end{equation}

We note that by (\ref{eq: chi}) we have  $\chi_b^2=\chi_b$ for all $b\in W$.  Since 
$W$ is a finite set, it follows that the set $X$ is also  finite.  
It thus follows from (\ref{eq: S}) and Remark \ref{rem: kernel} that 
 all maps $\chi$ in $ X$ are $p$-automatic.  
In particular, since by assumption $a({\bf n})=a_1({\bf n})\in W$, we deduce that the map 
$\chi_a$ is $p$-automatic. It follows that the set 
$${\mathcal Z}(f) = \left\{ {\bf n}\in \mathbb{N}^d \mid a({\bf n})=0\right\}$$ 
is a $p$-automatic set, which  ends the proof. 
\end{proof}

\section{Finite automata and effectivity}\label{section: ker}

In this section, we define a classical measure of complexity for $p$-automatic sets and 
we show how it can be used to prove effective results concerning such sets.  
We follow the presentation of \cite{Der}. 

\medskip

\begin{defn}{\em 
Let $S \subset \mathbb N^d$ be a $p$-automatic set and let denote by 
$K$ the $p$-kernel of $S$. We define the $p$-complexity of $S$ 
by 
$$
\mbox{ comp}_{\rm p}(S) := \mbox{ Card } K \, .
$$ }
\end{defn}

The aim of this section is to state and prove the following result. 

\begin{prop}\label{prop: eff}
Let $S \subset \mathbb N^d$ be a $p$-automatic set and suppose that there exists an explicit integer $N(S)$ such that 
$$
\mbox{\rm  comp}_{\rm p}(S) \leq N(S) \, .
$$
Suppose also that for every positive integer $n$ one can compute 
(in a finite amount of time) all the elements ${\bf s}\in S$ such that $\Vert {\bf s}\Vert\leq n$. 
Then the set $S$ can be effectively determined. Furthermore, the following 
properties are decidable. 

\medskip

\begin{itemize}
\item[{\rm (i)}] the set $S$ is empty.

\medskip

\item[{\rm (ii)}] the set $S$ is finite.

\medskip

\item[{\rm (iii)}] the set $S$ is periodic, that is, formed by the union of a finite set and 
of a finite number of ($p$-dimensional) arithmetic progressions.

\medskip

\end{itemize}
\noindent 
In particular, when $S$ is finite, one can find (in a finite amount of time) 
all its elements.
\end{prop}

\begin{rem}{\em
When we say that the set $S$ can be effectively determined, this means 
that there is an algorithm that produces in a finite amount of time a $p$-automaton 
that generates $S$. The format of the output is 
thus a $6$-tuple $\left(Q,\Sigma_p^d,\delta,q_0,\{0,1\},\tau\right)$, 
where $Q$ the set of states, $\delta:Q\times \Sigma_k^d\rightarrow Q$ 
is the transition function, $q_0$ is the initial state, and 
$\tau : Q\rightarrow \{0,1\}$ is the output function.  Furthermore, there exists an algorithm that allows 
one to determine in a finite amount of time whether $S$ is empty, finite or whether $S$ 
is formed by the union of a finite set and of a finite number of ($p$-dimensional) arithmetic 
progressions.} 
\end{rem}

We first make the important observation that for every positive integer $N$ there 
are only a finite number of $p$-automatic subsets of $\mathbb N^d$ whose 
$p$-complexity is at most $N$. 

\begin{lem}\label{lem: enum}
Let $N$ be a positive integer. Then there at most $N 2^N N^{pN}$ distinct $p$-automatic subsets 
of $\mathbb N^d$ whose $p$-complexity is at most $N$.  
\end{lem}

\begin{proof} In the definition of $p$-automatic sets in Section \ref{Salon}, 
we used $p$-automata that read the input ($d$-tuples of integers) starting from the most 
significant digits (the input is scanned from the left to the right). It is well known that using 
$p$-automata that read the input starting from the least significant digits 
(the input is scanned from the right to the left) leads to the same notion of $p$-automatic sets. 
Furthermore, it is known that for every $p$-automatic set $S$, there exists such a $p$-automaton 
for which the number of states is equal to the cardinality of the $p$-kernel of $S$. Such an automaton 
has actually the minimal number of states among all automata recognizing $S$ and reading the input from 
the right to the left (see for instance \cite{AS} or \cite{Der}).

Thus a $p$-automatic set $S\subseteq \mathbb N^d$ with $p$-complexity at most $N$ can be recognized 
by a $p$-automaton $\mathcal A$ (reading from the right to the left) with at most $N$ states. 
Let $Q:=\{Q_1,\ldots, Q_N\}$ denote the set of states of $\mathcal A$.   
To define $\mathcal A$, we must choose the initial state, the transition function from $Q\times \Sigma_p$ to $Q$, and the output function 
from $Q$ to $\{0,1\}$. We have at most $N$ choices for the initial state, 
at most $N^{pN}$ choices for the transition function, and 
at most $2^N$ choices for the output function.   
The result immediately follows. 
\end{proof}

\begin{lem}\label{lem: S1S2}
 Let $S_1,S_2\subseteq \mathbb N^d$ be $p$-automatic sets. 
Then the following hold.  
\begin{itemize}
\medskip
\item[$\bullet$] $
\mbox{ \rm comp}_{\rm p}(S_1\cap S_2) \leq 
\mbox{\rm  comp}_{\rm p}(S_1)\mbox{ \rm comp}_{\rm p}(S_2) .
$
\medskip
\item[$\bullet$] $
\mbox{\rm comp}_{\rm p}(S_1\cup S_2) \leq 
\mbox{\rm  comp}_{\rm p}(S_1)\mbox{ \rm comp}_{\rm p}(S_2) .
$
\medskip
\item[$\bullet$] 
$\mbox{\rm  comp}_{\rm p}((S_1\setminus S_2)\cup (S_2\setminus S_1)) \leq 
\mbox{ \rm comp}_{\rm p}(S_1)\mbox{ \rm comp}_{\rm p}(S_2) .
$
\medskip
\item[$\bullet$] 
$\mbox{\rm comp}_{\rm p}(S_1\setminus (S_1\cap S_2) )\leq 
\mbox{ \rm comp}_{\rm p}(S_1)\mbox{ \rm comp}_{\rm p}(S_2) .
$
\end{itemize}
\end{lem}

\begin{proof}
Given a set $S$ let us denote by ${\mathcal I}_S$ its indicator function. The proof follows 
from the fact that 
${\mathcal I}_{S_1\cap S_2} = {\mathcal I}_{S_1}\cdot {\mathcal I}_{S_2} $, 
${\mathcal I}_{S_1\setminus S_2} = {\mathcal I}_{S_1} \cdot(1-{\mathcal I}_{S_2}) $, 
${\mathcal I}_{S_1\cup S_2} = {\mathcal I}_{S_1} + {\mathcal I}_{S_2 } - {\mathcal I}_{S_1}\cdot {\mathcal I}_{S_2}$, 
${\mathcal I}_{(S_1\setminus S_2)\cup (S_2\setminus S_1)}= 
{\mathcal I}_{S_1}\cdot (1-{\mathcal I}_{S_2}) + {\mathcal I}_{S_2} \cdot(1-{\mathcal I}_{S_1}) $,  and 
${\mathcal I}_{S_1\setminus (S_1\cap S_2) }= 
{\mathcal I}_{S_1} \cdot (1-{\mathcal I}_{S_1} \cdot{\mathcal I}_{S_2}) $. 
\end{proof}

We will also use the following two results that can be easily proved as in \cite{Der}.

\begin{lem}\label{lem: min} Let $S\subseteq \mathbb N^d$ be a nonempty $p$-automatic set. 
Then 
$$
\min \left\{ \Vert{\bf s}\Vert \mid {\bf s} \in S\right\}  
\leq p^{\mbox{\rm comp}_{\rm p}(S) -2}\, .
$$
\end{lem}

\begin{lem}\label{lem: max} Let $S\subseteq \mathbb N^d$ be a finite $p$-automatic set. 
If ${\bf s}\in S$, then 
$$
\Vert {\bf s}\Vert \leq p^{\mbox{\rm comp}_{\rm p}(S) -2}\, .
$$
\end{lem}

We are now ready to prove Proposition \ref{prop: eff}. 

\begin{proof}[Proof of Proposition \ref{prop: eff}] 
Let $S\subseteq \mathbb N^d$ be a $p$-automatic set. Let 
us assume that one knows an effective bound $N(S)$ for the $p$-complexity of $S$ 
and that one can compute the initial terms of $S$.  Let us also assume that for every positive 
integer $n$ one can compute (in a finite amount of time) all the elements ${\bf s}\in S$ such that 
$\Vert {\bf s}\Vert\leq n$. 

\medskip

We first note that by Lemma \ref{lem: enum} 
there are only a finite number, say $r$, of $p$-automatic subsets of $\mathbb N^d$ 
with $p$-complexity at most $N(S)$. Going through the proof of Lemma \ref{lem: enum}, we can 
explicitly enumerate all these sets 
to get a collection $S_1,S_2,\ldots, S_r$. 

Now for each $S_i$, we can check whether $S=S_i$ as follows. Since both $S$ and $S_i$ 
have $p$-complexity at most $N(S)$, we infer from Lemma \ref{lem: S1S2} that  
$$
\mbox{ comp}_{\rm p}((S\setminus S_i)\cup (S_i\setminus S)) \leq 
\mbox{ comp}_{\rm p}(S)\mbox{ comp}_{\rm p}(S_i) \leq N(S)^2 \, .
$$
Thus, by Lemma \ref{lem: min}, the set $(S\setminus S_i)\cup (S_i\setminus S)$ is empty if and 
only if it has no element up to $p^{N(S)^2 -2}$. This implies that $S=S_i$ if and only if 
$$
S\cap \left\{ {\bf n}\in \mathbb N^d\mid \Vert{\bf n} \Vert\leq p^{N(S)^2 -2} \right\}=
S_i\cap \left\{ {\bf n}\in \mathbb N^d\mid \Vert{\bf n} \Vert \leq p^{N(S)^2 -2} \right\} \, .
$$
By assumption, this can be verified in a finite amount of time.

\medskip

\noindent{\bf (i).} Since the $p$-complexity of $S$ is at most $N(S)$, Lemma \ref{lem: min} implies that  
$S$ is empty if and only if 
$$
S \cap  \left\{ {\bf n}\in \mathbb N^d\mid \Vert{\bf n} \Vert\leq p^{N(S)^2 -2} \right\} =\emptyset \, .
$$
By assumption, this can be verified in a finite amount of time. 

\medskip

\noindent{\bf (ii).} Since the $p$-complexity of $S$ is at most $N(S)$, 
Lemma \ref{lem: max} implies that  $S$ is finite if and only if 
$$
S = S \cap \left\{ {\bf n}\in \mathbb N^d\mid \Vert{\bf n} \Vert\leq p^{N(S)^2 -2} \right\} \, .
$$
Set $S' :=  \left\{ {\bf n}\in \mathbb N^d\mid \Vert{\bf n} \Vert\leq p^{N(S)^2 -2} \right\}$. 
Thus $S$ is finite if and only if the set 
\begin{equation} \label{eq: empty}
S\setminus \left(S \cap S'\right) 
= \emptyset \, .
\end{equation}
On the other hand, it is easy to see that $S'$ is a $p$-automatic set with complexity at most 
$(p^{N(S)^2 -2}+1)^d$. By Lemma \ref{lem: S1S2}, we deduce that 
$$
\mbox{ comp}_{\rm p}\left(S\setminus \left(S \cap S'\right) \right) \leq \mbox{ comp}_{\rm p}(S) 
\mbox{ comp}_{\rm p}(S') \leq N(S) \left(p^{N(S)^2 -2}+1\right)^d  \, .
$$
This shows, using (i), that one can check whether Equality (\ref{eq: empty}) is satisfied 
in a finite amount of time. 

\medskip

\noindent{\bf (iii).} We have already shown that we can 
explitly determine a $p$-automaton that recognized $S$, since the $p$-complexity of $S$ is at most $N(S)$. Then a classical result of Honkala \cite{Hon86} 
shows that one can check whether such set is periodic, that is, whether $S$ is the union of a finite set 
and a finite number of ($p$-dimensional) arithmetic progressions.  

\medskip
Finally, to obtain all the elements of $S$ when $S$ is finite one can proceed as follows. First, one can 
check that $S$ is finite as in (ii). Once this has been done, one knows that $S$ is finite and thus Lemma 
\ref{lem: max} implies that 
$$
S=S \cap \left\{ {\bf n}\in \mathbb N^d\mid \Vert{\bf n} \Vert\leq p^{N(S)^2 -2} \right\}
$$
since $S$ has complexity at most $N(S)$. By assumption, all the elements of $S$ can thus 
be determined in a finite amount of time. This ends the proof. 
\end{proof}
\section{Proof of Theorem \ref{thm: eff}}\label{section: eff}

The aim of this section is to show how each step of the proof of Theorem \ref{thm: main} 
can be made effective.

\medskip

We first recall some notation.  
Given a polynomial $P(X)\in K[{\bf t}][X]$, we define the height of $P$ as 
the maximum of the degrees of the coefficient of $P$. 
The (naive) height of an algebraic power series 
$$f({\bf t})=\sum_{n \in \mathbb N} a({\bf n}) {\bf t}^{\bf n}\in K[[{\bf t}]]$$ 
 is then defined as the height of the minimal polynomial of $f$, or equivalently, 
 as the minimum of the heights of the nonzero polynomials $P(X)\in K[{\bf t}][X]$ 
 that vanishes at $f$.

\medskip

We first prove the following effective version of Ore's lemma.

\begin{lem}\label{lem: oref}
Let $s$ and $H$ be two positive integers and 
let $f({\bf t})\in K[[{\bf t}]]$ be an algebraic power series of degree at most $s$ and 
height at most $H$. Then there exist polynomials $Q_0,\ldots,Q_s\in K[{\bf t}]$ 
with degree at most $Hsp^s$ such that 
$$
\sum_{i=0}^s Q_i({\bf t})f({\bf t})^{p^i} \ = \ 0
$$
and $Q_0\not=0$. 
\end{lem}

In order to prove Lemma \ref{lem: oref}, we will need the following auxiliary result. 

\begin{lem}\label{lem: vi} Let $s$ be a natural number and let $V_0,\ldots, V_s$ be $s+1$ 
vectors in $K[{\bf t}]^s$ such that each coordinate has degree at most $r$.  Then there exist
$Q_0({\bf t}),\ldots ,Q_s({\bf t})$ in $K[{\bf t}]$ of degree at most $rs$, not all of which are zero, 
such that 
$$
\sum_{i=0}^s Q_iV_i=0 \, .
$$
\end{lem}

\begin{proof} Let $e$ denote the size of a maximally linearly independent subset of 
$V_0,\ldots ,V_s$.  By relabelling if necessary, we may assume that $V_0,\ldots ,V_{e-1}$ 
are linearly independent.   Let $A$ denote the $s\times e$ matrix whose $(j+1)$th column is 
$V_j$.  Then by reordering the coordinates of our vectors if necessary, we may assume that the 
$e\times e$ submatrix $B$ of $A$ obtained by deleting the bottom $d-e$ rows of $A$ is invertible.  
Let $V_s'$ denote the vector in $K[{\bf t}]^e$ obtained by deleting the bottom $d-e$ 
coordinates of $V_s$.  Then there is a unique vector $X$ that is solution to the matrix equation 
$$BX=V_s'\, .$$  Moreover, by Cramer's rule, the $i$th coordinate of $X$ is 
the polynomial $X_i$ defined by
$$X_i({\bf t}):=\det(B_i)/\det(B) \, ,$$ 
where $B_i$ is the $e\times e$ matrix obtained by replacing the 
$i$th column of $B$ by $V_s'$.  
For $0\le i \le e-1$, we set 
$$Q_i({\bf t}):=-\det(B_i) \,.$$ 
We also set  
$$Q_s({\bf t}):=\det(B) \, .$$ 
Since the entries of $B_i$ and $B$ are all polynomials of 
degree at most $r$, we obtain that these polynomials 
have degree at most $re\le rs$.   
Furthermore, by construction
$$
\sum_{i=0}^{e-1} X_i V_i = V_s \, .
$$
Letting $Q_i({\bf t})=0$ for $e\leq i<s$, we obtain that
$$\sum_{i=0}^s Q_iV_i=0$$ 
and each $Q_i$ has degree at most $rs$, as required.
\end{proof}   

\medskip

We are now ready to prove Lemma \ref{lem: oref}.

\begin{proof}[Proof of Lemma \ref{lem: oref}] 
By assumption, there exist polynomials $P_0({\bf t})$, $\ldots, P_s({\bf t})\in K[{\bf t}]$ 
of degree at most $H$ such that 
$$
\sum_{i=0}^sP_i({\bf t}) f({\bf t})^i = 0  
$$
and $P_s({\bf t}) \not=0$. 

Let $\mathcal V$ denote the $K({\bf t})$-vector space spanned by $1,f,\ldots ,f^{s-1}$.
For $1\leq i \leq s$, let $e_i$ denote the standard unit $d\times 1$ vector in $K({\bf t})^s$ whose 
$j$th coordinate is equal to the Kronecker delta $\delta_{ij}$.  
Then we have a surjective linear map $T: K({\bf t})^s \to \mathcal V$ in which 
we send the vector $e_i$ to $f^{i-1}$.  
Let $V=\sum_{i=1}^s T(e_i)e_i\in K({\bf t})^s$ and let 
\[M \ := \ \left( \begin{array}{cccccc} 0 & 0 & 0 & \cdots & 0 & -X_0({\bf t})\\
1 & 0 & 0 &  \cdots & 0 & -X_1({\bf t}) \\
0 & 1 & 0 & \cdots & 0 & -X_2({\bf t}) \\
\vdots & \vdots & \vdots & \ddots & \cdots & \vdots \\
0 & \cdots & 0& 1 & 0 & -X_{s-2}({\bf t}) \\
0& 0 &  \cdots & 0 & 1 & -X_{s-1}({\bf t}) \end{array}\right)  \in M_s(K({\bf t}))\, ,\] where 
$X_i({\bf t}):=P_i({\bf t})/P_s({\bf t})$ for $i=0,1,\ldots ,s-1$.
Then 
$$
T\left( M^n e_1\right)=f({\bf t})^n \, .
$$
Notice that $M^n=P_s({\bf t})^{-n}C_n$ where $C_n$ is a matrix in $M_s(K[{\bf t}])$ 
whose entries have degree at most $nH$.   
Then to find a relation of the form
$$
\sum_{i=0}^s Q_i({\bf t})f({\bf t})^{p^i} \ = \ 0 \, ,
$$ 
it is enough to find a vector
$$
[Q_0({\bf t}),  \ldots,Q_s({\bf t})]\in K[{\bf t}]^{1\times d}
$$ 
such that
\begin{equation}\label{eq: PQ}
P_s({\bf t})^{p^s} Q_0({\bf t}) e_1+P_s({\bf t})^{p^s-p}Q_1({\bf t})C_p e_1
+\cdots + Q_s({\bf t}) C_{p^s} e_1 = 0 \, .
\end{equation}

For $0\leq j\leq s$, we set 
\begin{equation}\label{eq: Vj}
V_j := P_s({\bf t})^{p^s-p^j}C_{p^j}e_1 \, .
\end{equation}
 We note that $V_j$ 
is a vector in $K[{\bf t}]^s$ such that each coordinate has degree 
at most $Hp^s$.   Then Lemma \ref{lem: vi} ensures the existence of polynomials 
$Q_0({\bf t}),\ldots ,Q_s({\bf t})$ in $K[{\bf t}]$ of degree at most $sHp^s$, 
not all of which are $0$, and such that
$$
\sum_{j=0}^s Q_j V_j =0 \, .
$$  
We deduce from Equations (\ref{eq: PQ}) and (\ref{eq: Vj}) that   
\begin{equation} 
\sum_{j=0}^s Q_j({\bf t}) f^{p^j}({\bf t}) =0 \, .\label{eq: 1}
\end{equation}

It thus remains to show that we can choose our polynomials 
$Q_0,\ldots, Q_s$ such that $Q_0$ is nonzero.   
To see this, we let $k$ denote the smallest index such that we have a relation of the form 
given in Equation (\ref{eq: 1}) with the degrees of $Q_0,\ldots ,Q_s$ all bounded above by $sHp^s$ 
and such that $Q_k$ is nonzero.  If $k$ is equal to zero, we are done.  

We now assume that $k>0$ 
and we argue by contradiction.   
Since $Q_k\not=0$, we infer from 
Equality (\ref{eq: fs}) that there exists 
a $d$-tuple ${\bf j}\in \Sigma_p^d$ such that $E_{\bf j}(Q_k)\not =0$.  Since 
$\sum_{i=k}^{s} Q_if^{p^i}=0$, we have 
$$
E_{\bf j}\left(\sum_{i=k}^{s} Q_if^{p^i}\right)= \sum_{i=k}^{s} E_{\bf j}\left(Q_if^{p^i}\right)
= \sum_{i=k}^{s} E_{\bf j}\left(Q_i\right) f^{p^{i-1}}= 0 \, .
$$
Furthermore, one can observe that, for $k\le i\le s$, the polynomial $E_{\bf j}(Q_i)$ has 
degree at most $sHp^s$. We thus obtain a new relation of the same type but for which the 
coefficient of $f^{p^{k-1}}$ is nonzero, which contradicts the minimality of $k$.  This ends the proof.
\end{proof}

\medskip

We are now ready to prove Theorem \ref{thm: eff}. 

\begin{proof}[Proof of Theorem \ref{thm: eff}]
We first explain our strategy. We assume that $f({\bf t}) \in K[[{\bf t}]]$ is an algebraic function 
and that we know an explicit polynomial $P(X)\in  K[{\bf t}][X]$ that vanishes at $f$.  
Note that from the equation $P(f)=0$, one can obviously derive explicit effective bounds 
of the degree and of the height of $f$. Then we will show how the proof of Theorem \ref{thm: main} 
allows us to derive an effective upper bound for $\mbox{comp}_{\rm p}({\mathcal Z}(f))$.  
It will thus follows from the results of Section \ref{section: ker} that one can effectively determined the  
set ${\mathcal Z}(f)$ only by looking at the first coefficients of $f$ (which can be computed in a finite 
amount of time by using the equation $P(f)=0$). 

\medskip

Let us assume that the degree of $f$ is bounded by $s$ and that the height of $f$ is bounded by $H$. 
In order to get an effective upper bound for $\mbox{comp}_{\rm p}({\mathcal Z}(f))$, 
we have to give effective upper bounds 
for the cardinality of the sets $U,V,W$ and $X$ introduced all along the proof of Theorem \ref{thm: main}.

\medskip

\noindent{\bf Step 1.} In this first step we show how to obtain an effective  
upper bound for the dimension $m$ of the $K$-vector space spanned by $\Omega(f)$. 
We then deduce an effective upper bound for the cardinality of the $\mathbb F_p$-vector space 
$U$.

This can be deduced from our effective version of Ore's lemma. 
Indeed, by Lemma \ref{lem: ore}, one can find polynomials $Q_0,\ldots,Q_s\in K[{\bf t}]$ with degree at most $Hsp^s$ such that 
$$
\sum_{i=0}^s Q_i({\bf t})f({\bf t})^{p^i} \ = \ 0
$$
and $Q_0\not=0$.  
We set ${\tilde f}:=Q_0^{-1}f$.   
Then 
\begin{equation}\label{eq: geff}
{\tilde f}= \sum_{i=1}^s R_i {\tilde f}^{p^i} \, ,
\end{equation}
where $R_i= -Q_i Q_0^{p^i-2}$. 
Then each $R_i$ has degree at most $Hsp^s(p^i-1)$. 
Set $M := Hsp^s(p^s-1)$ and 
\begin{equation}\label{eq: Heff}
{\mathcal H} := \left\{ h \in K(({\bf t})) \mid h = \sum_{i=0}^s S_i{\tilde f}^{p^i} \mbox{ such that } 
S_i \in K[{\bf t}] \mbox{ and } \deg S_i \leq M \right\} \, .
\end{equation}
Furthermore, $\mathcal H$ is a $K$-vector space of dimension at most 
$$(s+1) {M+d \choose M}\, .$$ 
Just as in the proof of Lemma \ref{lem: SW}, one can check that $f$ belongs to $\mathcal H$ and 
that ${\mathcal H}$ is closed under the action of $\Omega$.  
It follows that the $K$-vector space spanned by $\Omega(f)$ is contained in ${\mathcal H}$. 
There thus exists an effective constant $N_0 := (s+1){M+d \choose M}$ 
such that the $K$-vector space spanned by $\Omega(f)$ has dimension 
\begin{equation}\label{eq: N0}
m \leq N_0  \, .
\end{equation}
We recall that $K_0$ denotes  the subfield of 
$K$ generated by the coefficients of $f_1\ldots , f_m$ 
and all the elements $\lambda(i,{\bf j},k)$ $1\le i,k\le m$ and ${\bf j} \in \{0,1,\ldots ,p-1\}^d$, and that 
$U$ is defined as the finite-dimensional $\mathbb{F}_p$-vector subspace of 
$K_0$ spanned by the elements $\lambda(i,{\bf j},k)$, $1\le i,k\le m$ and ${\bf j} \in \{0,1,\ldots ,p-1\}^d$, 
and by $1$.  We thus deduce from (\ref{eq: N0}) that there exist an effective upper bound 
$N_1:= p^{1+p^dN_0^2}$ such that 
\begin{equation}\label{eq: N1}
\mbox{Card}(U) \leq N_1  \, .
\end{equation}

\medskip

\noindent{\bf Step 2.}  From Derksen's proposition (Proposition \ref{AB:proposition:derksen}),  
we know that there exists a finite-dimensional $\mathbb{F}_p$-vector 
subspace $V$ of $K_0$ containing $U$ such that $\pi_i(VU)\subseteq V$ for $1\leq i \leq r$. 
In this second step, we show how to obtain an effective upper bound for the cardinality of such 
a vector space $V$. 

\medskip


In the proof of Lemma \ref{lem: fg}, we have shown that $K_0$ is a finitely generated field 
extension of $\mathbb{F}_p$ that can be generated by the 
$\lambda(i,{\bf j},k)$ and the coefficients of a finite number of some explicit polynomials.
We write 
$$K_0=\mathbb{F}_p(X_1,\ldots ,X_r)(a_1,\ldots ,a_s)\, ,$$ 
where $X_1,\ldots ,X_r$ are 
algebraically independent over $\mathbb{F}_p$ and $a_1,\ldots , a_s$ form a basis for 
$K_0$ as an $\mathbb{F}_p(X_1,\ldots ,X_r)$-vector space; moreover, we may assume that 
for each $i$ and $j$, we have $a_ia_j$ is an $\mathbb{F}_p(X_1,\ldots ,X_r)$-linear combination 
of $a_1,\ldots ,a_s$ in which the numerators and denominators of the coefficients are 
polynomials in $\mathbb{F}_p[X_1,\ldots ,X_r]$ whose degrees are uniformly bounded by 
some explicit constant $N_2$.

Let $T_1,\ldots,T_n$ denote such a set of generators of $K_0$ with the following properties.
\begin{enumerate}
\item[{\rm (i)}] $T_i=X_i$ for $i\le r$.
\item[{\rm (ii)}] $T_{r+j}=a_j$ for $j\le s$.
\item[{\rm (iii)}] $T_n=1$.
\item[{\rm (iv)}] $\{ T_1,\ldots,T_n\}$ contains all the 
$\lambda(i,{\bf j},k)$.  
\end{enumerate}
Note that from Step 1 and the proof of Lemma \ref{lem: fg} we can obtain an explicit 
upper bound for the integer $n$. 

An easy induction shows that for $1\le j\le s$, $a_j^p$ is an $\mathbb{F}_p(X_1,\ldots ,X_r)$-linear 
combination of $a_1,\ldots ,a_s$ in which the coefficients are rational functions whose 
numerators and denominators have degrees uniformly bounded by 
\begin{equation}\label{N3}
N_3 := N_2\left(2s^{p-2}+\frac{s^{p-2}-s}{s-1}\right)\, .
\end{equation} 
We now regard $K_0$ as an $s$-dimensional $\mathbb{F}_p(X_1,\ldots ,X_r)$-vector space.  
Then we may regard the $\mathbb{F}_p(X_1,\ldots ,X_r)$-span of $a_1^p,\ldots ,a_s^p$ 
as a subspace of $$\mathbb{F}_p(X_1,\ldots ,X_r)^s$$ spanned by $s$ vectors whose 
coordinates are rational functions whose numerators and denominators have degrees 
uniformly bounded by $N_3$.  We can effectively compute the dimension of this space 
and a basis.  We let $t$ denote the dimension of this vector space and by relabelling 
if necessary, we may assume that $a_1^p,\ldots ,a_t^p$ form a basis.  
Then there exist $\ell_1,\ldots ,\ell_{s-t}$ such that 
$\{a_1^p,\ldots ,a_t^p, a_{\ell_1},\ldots ,a_{\ell_{s-t}}\}$ forms a basis for $K_0$ as a 
$\mathbb{F}_p(X_1,\ldots ,X_r)$-vector space. 
Moreover, using Cramer's rule, we can express each $a_j$ as a 
$\mathbb{F}_p(X_1,\ldots ,X_r)$-linear combination of 
$a_1^p,\ldots ,a_t^p, a_{\ell_1},\ldots ,a_{\ell_{s-t}}$ in which the numerators and denominators 
have degrees uniformly bounded by 
\begin{equation}\label{N4}
N_4:= 2N_3st \, .
\end{equation}
To see this, let $\phi:K_0\to {\mathbb F}_p(X_1,\ldots ,X_r)^s$ be the 
$\mathbb{F}_p(X_1,\ldots ,X_r)$-linear isomorphism which sends $a_i$ to 
the vector with a $1$ in the $i$th coordinate and zeros in all other 
coordinates.  Let $A$ denote the $s\times s$ matrix whose $i$th row is 
equal to $\phi(a_i^p)$ for $i\le t$ and is equal to $\phi(a_{\ell_{t-i}})$ for $i>t$.  Then 
the entries of $A$ are rational functions whose numerators and 
denominators have degrees that are uniformly bounded by $N_3$.  Note 
that expressing $a_j$ as a $\mathbb{F}_p(X_1,\ldots ,X_r)$-linear 
combination of $a_1^p,\ldots ,a_t^p,a_{\ell_1},\ldots ,a_{\ell_{t-s}}$ is 
the same as solving the matrix equation
$$A{\bf x} \ = \ {\bf e}_j\, ,$$ where ${\bf e}_j$ is the vector whose $j$th 
coordinate is $1$ and whose other coordinates are $0$.  By Cramer's rule, 
the $i$th coordinate of ${\bf x}$ is a ratio of two $s\times s$ 
determinants, each of which have entries which are rational functions in 
$\mathbb{F}_p(X_1,\ldots ,X_r)$ whose numerators and denominators have 
degrees that are uniformly bounded by $N_3$, and such that the bottom 
$s-t$ rows consist of constants.  Note that the determinant of an $s\times 
s$ matrix whose entries are rational functions is a rational function; 
moreover, we can take the denominator to be the product of the 
denominators of the entries.  Since our matrices have a total of $st$ 
entries which are not constant, we obtain a bound of $N_3st$ for the 
degrees of the denominators of our determinants.  It is easy to check that 
this bound applies to the degrees of the numerators as well.  When we take a 
ratio of these determinants, this can at most double this bound on the 
degrees of the numerators and denominators.  Thus we can express each 
$a_j$ as a $\mathbb{F}_p(X_1,\ldots ,X_r)$-linear combination of
$a_1^p,\ldots ,a_t^p,a_{\ell_1},\ldots ,a_{\ell_{t-s}}$ in which the 
degrees of the numerators and denominators are uniformly bounded by 
$2N_3st$, as claimed.

Notice that  
$$
S:=\left\{ T_1^{i_1}\cdots T_{n}^{i_n}~\mid~0\le i_1,\ldots ,i_{n}<p\right\}
$$
spans $K_0$ as a $K_0^{\langle p\rangle}$-vector space.
Observe also that every polynomial $Q\in \mathbb F_p[T_1,\ldots,T_n]$ can be 
decomposed as 
\begin{equation}\label{eq: Qf}
Q = \sum_{f\in S} Q_f^p f \,,
\end{equation}
where the $Q_f$ are polynomials in $\mathbb F_p[T_1,\ldots,T_n]$ of degree 
at most $\lfloor \deg Q/p\rfloor$.  

 Let us choose $S_0$ to be the subset of $S$ corresponding to the monomials 
 from the set formed by the union of 
$$\left\{X_1^{i_1}\cdots X_r^{i_r} \mid 0\le i_1,\ldots ,i_r<p\right\}$$ 
and 
$$\left\{X_1^{i_1}\cdots X_r^{i_r} a_{\ell_j}\mid 0\le i_1,\ldots ,i_r<p, 1\le j\le s-t\right\}\, 
.$$ 
Then $S_0$ is a basis for $K_0$ as a $K_0^{\langle p\rangle}$-vector space.  
Thus for $T_1^{i_1}\cdots T_n^{i_n}\in S$, we have 
$$
T_1^{i_1}\cdots T_n^{i_n}=\sum_{h\in S_0} \alpha_{h,i_1,\ldots ,i_n}^p h
$$
for some coefficients $\alpha_{h,i_1,\ldots ,i_n}\in K_0$.
We may pick some nonzero polynomial $H(T_1,\ldots ,T_n)$ such that
\begin{equation}
\label{eq: xxx}
H(T_1,\ldots ,T_n)^pT_1^{i_1}\cdots T_n^{i_n}
=\sum_{h\in S_0} A_{h,i_1,\ldots ,i_n}^ph \, ,
\end{equation} 
where 
$$A_{h,i_1,\ldots ,i_n} \in \mathbb{F}_p[T_1,\ldots ,T_n]
$$ 
for all 
$$(h,i_1,\ldots ,i_n)\in S_0\times \{0,1,\ldots ,p-1\}^n\, .
$$
We let 
\begin{equation}
\label{eq: xxy}
M':=\max\, \left\{ \deg H, \,   \deg  A_{h,i_1,\ldots ,i_n} \right\}
\end{equation}
where the maximum is taken over all 
$$
(h,i_1,\ldots ,i_n)\in S_0\times \{0,1,\ldots ,p-1\}^n \,.
$$
We claim that it is possible to obtain an effective upper bound for $M'$, once the set of generators 
and the basis are known.   To see this, note that
we write $T_i=\sum_{j=1}^s \phi_{i,j}(X_1,\ldots ,X_r)a_s$, where each $\phi_{i,j}$ is a rational function in $X_1,\ldots ,X_r$, where we assume that the degrees of the numerators and denominators 
of the $\phi_{i,j}$ are uniformly bounded by some explicit constant $N_5$.

Note that by construction, a monomial $T_1^{i_1}\cdots T_n^{i_n}$ with 
$0\le i_1,\ldots ,i_n<p$ is an $\mathbb{F}_p(X_1,\ldots ,X_r)$-linear combination of 
$a_1,\ldots ,a_s$ in which the coefficients are rational functions whose numerators 
and denominators have degrees uniformly bounded by 
\begin{equation}\label{N6}
N_6 := (N_2+N_5)(p-1)ns^{2(p-1)n}\, .
\end{equation}
To see this, we claim more generally that a monomial $T_1^{j_1}\cdots 
T_n^{j_n}$ can be expressed as a $\mathbb{F}_p(X_1,\ldots ,X_r)$-linear 
combination of $a_1,\ldots ,a_s$ in which the coefficients are rational 
functions whose numerators and denominators have degrees uniformly bounded 
by $$(N_2+N_5)(j_1+\cdots +j_n)s^{2(j_1+\cdots +j_n)} \, .$$  
We prove this by 
induction on $j_1+\cdots +j_n$.  When $j_1+\cdots +j_n=1$, the claim is 
trivially true. So we assume that the claim holds whenever $j_1+\cdots 
+j_n<k$ and we consider the case that $j_1+\cdots +j_n=k$.  Then $j_i\ge 
1$ for some $i$. 
Thus we may write
$$T_1^{j_1}\cdots T_n^{j_n}=T_i \cdot T_1^{j_1}\cdots T_i^{j_i-1}\cdots 
T_n^{j_n}\, .$$  By the inductive hypothesis,
$$T_1^{j_1}\cdots T_i^{j_i-1}\cdots
T_n^{j_n} = \sum_{\ell=1}^s \psi_{\ell} a_{\ell}\, ,$$ where each
$\psi_{\ell}$ is a rational function whose numerator and denominator have 
degrees bounded by $(N_2+N_5)(k-1)s^{2k-2}$. 
Thus
\begin{eqnarray*}
&~& T_i \cdot T_1^{j_1}\cdots T_i^{j_i-1}\cdots
T_n^{j_n} \\
&=& \left( \sum_{j=1}^s \phi_{i,j} a_j \right)\left( \sum_{\ell=1}^s \psi_{\ell} 
a_{\ell} \right)
\\
&=& \sum_{1\leq j,\ell\leq s} (\phi_{i,j}\psi_{\ell}) a_ja_{\ell}\, .
\end{eqnarray*}
Recall that by assumption each $a_ja_{\ell}$ is a 
$\mathbb{F}_p(X_1,\ldots, X_r)$-linear combination of $a_1,\ldots ,a_s$ in 
which the degrees of the numerators and denominators are uniformly bounded 
by $N_2$.  Thus the coefficient of each $a_j$ is a linear combination 
consisting of $s^2$ terms whose numerators and denominators have degrees 
bounded by $N_5+(N_2+N_5)(k-1)s^{2k-2}+N_2$ and hence can be expressed as 
a rational function whose numerator and denominator have degrees bounded 
by $(N_2+N_5)(1+(k-1)s^{2k-2})s^2\le (N_2+N_5)ks^{2k}$.  This gives the bound (\ref{N6}), as claimed.

Then we may write 
$$T_1^{i_1}\cdots T_n^{i_n} = \sum_{j=1}^s \frac{C_j(X_1,\ldots ,X_r)}{D(X_1,\ldots ,X_r)^p} a_j\, ,$$ 
where
$C_1,\ldots , C_s,D$ are polynomials of degree at most $N_6sp$.  
 Furthermore, we showed in (\ref{N4}) that each $a_j$ can be written as a 
 $\mathbb{F}_p(X_1,\ldots ,X_s)$-linear combination of 
 $\{a_1^p,\ldots ,a_t^p, a_{\ell_1},\ldots, a_{\ell_{s-t}}\}$ in which the coefficients have numerators 
 and denominators  uniformly bounded by $N_4$.  Thus we may write
 $$T_1^{i_1}\cdots T_n^{i_n} = \sum_{j=1}^t \frac{\widehat{C}_j(X_1,\ldots ,X_r)}{\widehat{D}(X_1,\ldots ,X_r)^p }a_j^p + 
 \sum_{j=1}^{s-t}  \frac{\widehat{C}_j(X_1,\ldots ,X_r)}{\widehat{D}(X_1,\ldots ,X_r)^p} a_{\ell_j}\, ,$$ 
 where $\widehat{C}_1,\ldots ,\widehat{C}_{s},\widehat{D}$ have degrees bounded 
 by $$(N_4+N_6)sp\,. $$  
 Since $S_0$ forms a basis for $K_0^{\langle p\rangle}$, 
this ensures that 
\begin{equation}\label{eq: M}
M' \leq (N_4+N_6)sp+p\, .
\end{equation}

Now, we set 
\begin{equation}\label{eq: V0}
U_0:=\mathbb{F}_p H^{-1}+\sum_{j=1}^n \mathbb{F}_p T_j\, .
\end{equation}
Since $\left\{1\right\}\cup \left\{\lambda_{i,{\bf j},k} 
\mid 1\le i,k\le m, \, {\bf j} \in \{0,1,\ldots ,p-1\}^d \right\} 
\subseteq \left\{ T_1,\ldots ,T_n \right\}$, 
we have 
\begin{equation}\label{eq: VV0}
U \subseteq U_0 \, .
\end{equation}

Let $k$ be a positive integer. We infer from (\ref{eq: V0}) that 
$U_0^k$ is contained in the $\mathbb F_p$-vector space spanned by the set 
$$
{\mathcal K}:=\left\{H^{-j_0}T_1^{j_1}\cdots T_n^{j_n}~\mid~ \sum_{i=0}^nj_i \leq k \right\} \, .
$$ 
Then, every element $L$ of $\mathcal K$ can be written as 
\begin{equation}\label{eq: K}
L = H^{-pi_0}(H^{\ell} T_1^{j_1}\cdots T_n^{j_n}) 
=:H^{-p(i_0+1)}H^pQ \, 
\end{equation}
where $Q:= H^{\ell}T_1^{j_1}\cdots T_n^{j_n}$, $0\le \ell<p$, $0\leq i_0 \leq \lfloor k/p\rfloor$ 
and $\sum_{i=1}^nj_i \leq k-pi_0$. Thus 
$Q$ is a polynomial in $\mathbb F_p[T_1,\ldots,T_n]$ 
of degree at most $(p\deg H + k-pi_0)$.  By (\ref{eq: Qf}), 
$Q$ can be decomposed as 
$$
Q = \sum_{f \in S} Q_f^p f\, ,
$$
where $Q_f$ are polynomials in $\mathbb F_p[T_1,\ldots,T_n]$ 
of degree at most $\deg H + \lfloor k/p\rfloor -i_0$. We deduce that 
$$
H^{-(i_0+1)}Q_f \in U_0^{M' + \lfloor k/p\rfloor+1} \, .
$$
Thus we have
\begin{equation}\label{eq: HQ}
H^pQ = \sum_{f \in S} Q_f^p (H^pf)\, .
\end{equation}
Furthermore, by assumption, for $f\in S$
\begin{equation}\label{eq: Hf}
H^p f \in \bigoplus_{h\in S_0} (U_0^{M'})^{\langle p\rangle} h \, .
\end{equation}
We infer from (\ref{eq: K}), (\ref{eq: HQ}) and (\ref{eq: Hf}) that 
$$
L \in \bigoplus_{h\in S_0}  (U_0^{2M'+\lfloor k/p\rfloor+1})^{\langle p\rangle} h
$$
and thus 
$$U_0^k \subseteq \bigoplus_{h\in S_0}  (U_0^{2M'+\lfloor k/p\rfloor+1})^{\langle p\rangle} h\, .$$
Let $k_0:=\left\lfloor  2(M'+1)p/(p-1)\right\rfloor +1$ and set  $V:=U_0^{k_0-1}$. 
This choice of $k_0$ implies that $\pi_i(VU)\subseteq V$. Furthermore, $U\subseteq V$ 
and the cardinality of $V$ is bounded by $\mbox{ Card }U_0^{k_0-1}\leq (\mbox{ Card }U_0)^{k_0-1}\leq 
p^{(n+1)(k_0-1)}$. 
Since one could find an effective upper bound for $n$ and since Inequality (\ref{eq: M}) provides an 
effective upper bound for $M'$ (and thus for $k_0$), we obtain that 
there exists an effective constant $N_7$ such that 
$$
\mbox{ Card }V \leq N_7\,.
$$ 

\medskip

\noindent{\bf Step 3.} In this last step, we show how to derive from Step 2 effective upper bounds for the cardinality of the sets $W$ and $X$, from which we will finally deduce an effective upper bound for 
$\mbox{Comp}_{\rm p}{\mathcal Z}(f)$. 

\medskip

We just show that it is possible to  get an effective upper bound $N_7$ 
for the cardinality of the 
$\mathbb F_p$-vector space $V$. We now recall that the set $W$ is defined by  
\begin{displaymath}
W:=Va_1+\cdots +Va_m \,.
\end{displaymath}
We thus have 
$\mbox{ Card } W \le (\mbox{ Card } V)^m$, and we infer from (\ref{eq: N0}) that  there exists an 
effective constant  $N_8 := N_7^{N_0}$ such that 
\begin{equation}\label{eq: N8}
\mbox{ Card } W \le N_8  \, .
\end{equation}
 
We recall that given a map $b:\mathbb{N}^d\to K_0$, the map  
$\chi_b:\mathbb{N}^d\to \{0,1\}$ is defined by 
\begin{equation}\label{eq: chieff}
 \chi_b({\bf n}) \ = \ 
 \left\{ 
 \begin{aligned} 
 0 & \;{ \rm if } \; b({\bf n})\mathbb \neq 0 \,  \\
1 &\;{ \rm if } \;~b({\bf n})=0 \, . 
\end{aligned} 
\right. 
\end{equation}
We recall that the set $X$ is defined by 
\begin{equation*}
X := \{ \chi_{b_1}\cdots \chi_{b_t}~\mid~t\ge 0, b_1,\ldots , b_t\in W\} \, .
\end{equation*} 
Since  $\chi_b^2=\chi_b$ for all $b\in W$ and since the product of 
maps $\chi_b$ is commutative, we get that 
$$
\mbox{ Card } X \le 2^{ \mbox{ Card } W } \, .
$$
Thus we infer from (\ref{eq: N8}) the existence of an effective constant  $N_9 := 2^{N_8}$ 
such that 
$$
\mbox{ Card } X \leq N_9 \, .
$$ 
On the other hand, the proof of Theorem \ref{thm: main} shows that 
the $p$-kernel of  ${\mathcal Z}(f)$ is contained in $X$, which implies that 
$$
\mbox{comp}_{\rm p}({\mathcal Z}(f)) \leq N_9\, . 
$$
This ends the proof.
\end{proof}

\section{Concluding remarks}
\label{conclude}
We end our paper with a few comments. 
We note that Derksen \cite{Der} also proved a refinement  of his 
Theorem \ref{thm:derksen}. Let us state his result.  
Let $p$ be a prime number and let $q$ be a power of $p$.  Given $c_0,\ldots ,c_d\in
 \mathbb{Q}^*$ with $(q-1)c_i\in \mathbb{Z}$
  for $i\in \{1,\ldots ,d\}$ and
   $c_0+\cdots + c_d\in \mathbb{Z}$, we define
$$
\tilde{S}_q(c_0,\ldots ,c_d):=\left\{c_0+c_1 q^{i_1}+\cdots + c_d q^{i_d}~\mid~i_1,\ldots ,i_d\ge 0\right\} 
$$
and we take
$$
S_q(c_0,\ldots ,c_d):=\mathbb{N}\cap \tilde{S}_q(c_0,\ldots ,c_d)\, .
$$
If $c_i>0$ for some $i\in \{1,\ldots ,d\}$, we say that 
$S_q(c_0,\ldots ,c_d)$ is an elementary $p$-nested set of order $d$.  
We say that a subset of the natural numbers is $p$-nested of order $d$ if it is a finite union of 
elementary $p$-nested sets of order at most $d$ with at least one set having order exactly $d$.  
We then say that a subset of the natural numbers is $p$-normal of order $d$ if it is, up to a finite set, 
the union of a finite number of arithmetic progressions along with a $p$-nested set of order $d$.  
Derksen \cite[Theorem 1.8]{Der} proved that the zero set of a linear recurrence of order $d$ 
is a $p$-normal set of order at most equal to $d-2$. Of course, this refines the fact 
that such a set is $p$-automatic.

We already observed in the introduction that Theorem \ref{thm: main} 
is in some sense best possible since any $p$-automatic subset of $\mathbb N^d$ 
can be obtained as the set of vanishing coefficients of an algebraic 
power series  in $\mathbb F_p[[t_1,\ldots,t_d]]$.  However, one might hope that a  
refinement, involving a reasonable version of multidimensional 
$p$-normal set, could hold if we restrict our attention to multivariate rational functions. This is actually 
not the case. Even for bivariate rational functions over finite fields, the set of vanishing 
coefficients can be rather pathological. Indeed, Furstenberg \cite{Fur}  
showed that the diagonal of a multivariate rational power series with coefficients in a field of positive characteristic is an algebraic power series in one variable\footnote{Deligne \cite{De} generalized this result 
to diagonals of algebraic power series with coefficients in a field of positive characteristic.}. Moreover, the converse holds for any field: 
any one variable algebraic power series can be obtained as the diagonal of a bivariate rational power series\footnote{This result is essentially proved in \cite{Fur}. Denef and Lipshitz \cite{DL} actually obtained the following  stronger result:  any algebraic power series in $n$ variables with coefficients in an arbitrary field can be obtained as the diagonal of a  rational power series in $2n$ variables.}.  
In light of Christol's theorem, this implies in particular that 
any $p$-automatic subset of $\mathbb N$ can be realized as the diagonal of the set 
of vanishing coefficients of a bivariate rational power series with coefficients in $\mathbb F_p$.

Nevertheless, we may imagine that a similar refinement of Theorem \ref{thm: main}  
does exist for the special rational functions that appear 
in the Diophantine applications given in Sections \ref{decidability}, \ref{sunit} and \ref{MLT}. 
Finally, since these applications only involve multivariate rational functions, it would be interesting to 
find natural Diophantine problems that would involve some sets of vanishing coefficients of algebraic irrational 
multivariate power series.

\bigskip

\noindent{\bf \itshape Addendum.}\,\,--- 
During the last stage of the writing of this paper, 
the authors learned about a related work (though not written in terms of automata) 
of Derksen and Masser \cite{DM}. 
These authors obtain in particular strong effective results for the 
general $S$-unit equations over fields of positive characteristic and more generally for the 
Mordell--Lang theorem, in the special case of linear subvarieties of $G_m^n(K)$ for fields $K$ of 
positive characteristic.  

\bigskip

\noindent{\bf \itshape Acknowledgement.}\,\,--- The authors would like to thank Jean-Paul Allouche, 
David Masser and the anonymous referees for their useful remarks. They are also indebted to Ga\"el R\'emond for his interesting comments  concerning the relation between Theorem 4.1 and Corrolary 4.1. 
The first author is also most grateful 
to Aur\'elie and Vadim for their constant patience and support during the preparation of this paper. 

\end{document}